\documentclass[final,1p,times,authoryear]{elsarticle}
\usepackage{amsmath,amsthm,amssymb,mathtools}
\usepackage{bm}
\usepackage{microtype}
\usepackage{enumitem}
\usepackage{hyperref}
\usepackage{xcolor}
\hypersetup{colorlinks=true,linkcolor=blue,citecolor=blue,urlcolor=blue}
\newtheorem{theorem}{Theorem}
\newtheorem{proposition}[theorem]{Proposition}
\newtheorem{lemma}[theorem]{Lemma}
\newtheorem{corollary}[theorem]{Corollary}
\theoremstyle{definition}
\newtheorem{definition}[theorem]{Definition}

\newtheorem{remark}[theorem]{Remark}

\newcommand{\R}{\mathbb{R}}
\newcommand{\C}{\mathbb{C}}

\newcommand{\Z}{\mathbb{Z}}
\newcommand{\E}{\mathbb{E}}

\DeclareMathOperator{\Rea}{Re}

\DeclarePairedDelimiter{\inner}{\langle}{\rangle}

\journal{}

\begin{document}

\begin{frontmatter}

\title{Chirp-Induced Non-Separable Gabor Windows on $\mathbb{R}^d$}

 \author[label1]{Lorenzo De Leonardis}

\affiliation[label1]{organization={Department of International School for Advanced Studies (SISSA) },
             addressline={Via Bonomea 265},
             city={Trieste},
             postcode={34136},
             country={Italy.}}
\author[label2]{Alessandro Mazzoccoli}
\affiliation[label2]{organization={Department of Economics, Roma Tre University},
             addressline={Via Silvio d'Amico 77},
             city={Rome},
             postcode={00145},
             country={Italy.}}
\author[label2]{Pierluigi Vellucci}

\begin{abstract}
We construct an explicit class of non-separable Gabor windows on $L^2(\R^d)$ by applying chirp deformations to tensor-product dual pairs on separable lattices. Starting from one-dimensional dual Gabor frames, we first obtain separable higher-dimensional dual pairs by tensorization. We then transport these systems through the unitary chirp operator $U_C f(x)=e^{\pi i x^T Cx}f(x)$ and the associated phase-space shear, obtaining Gabor systems on lower block-triangular lattices of the form
$
\Lambda_{A,B,C}=\{(Ak,CAk+B\ell):k,\ell\in\Z^d\}.
$
{The construction is governed by a single unitary conjugacy identity for the mixed Gabor
reconstruction operator. As consequences, compact support, smoothness, frame bounds,
exact duality, canonical duals, and approximate-duality errors are transported without loss.} {We also record a covariance-level diagnostic, depending on a chosen STFT reference window,
which describes the second-order time--frequency tilt induced by the chirp. This diagnostic is
separated from the unchanged frame stability constants and is not used here to claim improved
sparsity, denoising, conditioning, or computational complexity.}
\end{abstract}

\begin{keyword}
Gabor frames \sep non-separable windows \sep chirp deformations \sep dual windows \sep time-frequency analysis
\end{keyword}

\end{frontmatter}

\section{Introduction}

The construction of explicit Gabor frames in higher dimension involves two competing requirements.
On the one hand, tensor-product constructions on separable lattices provide a tractable way to lift
one-dimensional windows and dual windows to $L^2(\R^d)$. On the other hand, many genuinely
multivariate signals exhibit geometric features that are not aligned with the coordinate axes, and
therefore call for time-frequency atoms with cross-coordinate coupling. The central question addressed
in this paper is whether one can pass from separable Gabor systems to genuinely non-separable
higher-dimensional systems while retaining explicit dual windows, compact support, smoothness, and
quantitative reconstruction control.

{The answer developed below is constructive. Starting from dual Gabor pairs on product lattices, we
apply a chirp deformation and the associated shear of phase space. This produces chirped windows
on sheared lower block-triangular lattices, while preserving the frame bounds, exact duality,
approximate duality, and the reconstruction estimates inherited from the separable model.
Under an explicit off-diagonal condition on the chirp matrix, the resulting windows are genuinely
non-separable.}

\subsection{Why non-separable windows? Geometry versus noise and computational baselines}\label{subsec:why_nonseparable}

A natural question is why one should introduce \emph{non-separable} Gabor windows in $\R^d$, and whether
non-separability is expected to improve (or worsen) reconstruction in the presence of noise.
The viewpoint adopted in this paper is deliberately \emph{structural}: our chirp-based construction is
designed to change the \emph{geometry} of the time-frequency atoms (their orientation and coupling across
coordinates) while \emph{not} deteriorating the standard stability and noise-propagation constants
attached to the underlying frame reconstruction.

Separable windows of the form $g(x)=\prod_{j=1}^d g_j(x_j)$ on separable lattices
$A\Z^d\times B\Z^d$ provide a robust and highly tractable baseline.
They allow one to lift one-dimensional constructions to higher dimension by tensorization, yielding explicit
dual pairs with controllable regularity (e.g.\ compact support and smoothness) and transparent proofs.
From a practical standpoint, separability also simplifies implementation: many operations factor into
lower-dimensional components, leading to reduced computational complexity and memory footprint.

Despite these advantages, separability imposes a rigid axis-aligned structure.
In multivariate signals and data, relevant features are often \emph{anisotropic} and \emph{coupled} across
coordinates (e.g.\ oriented patterns, oblique structures, and correlated variations). For such phenomena,
axis-aligned atoms may require many coefficients to represent a single geometric feature, whereas atoms that
incorporate cross-coordinate coupling can yield representations that are more faithful to the intrinsic
geometry of the signal. In this sense, non-separability is not introduced here to claim a generic improvement
of global stability constants, but to provide a controlled way to ``tilt'' and couple time-frequency atoms in
$\R^d$ while retaining full analytical tractability.

It is important to separate two distinct issues:
\begin{enumerate}[label=\textup{(\roman*)}]
\item \emph{Linear reconstruction stability} (in the sense of frames), which governs how coefficient perturbations
propagate into the reconstructed signal;
\item \emph{Nonlinear denoising performance} (e.g.\ thresholding/shrinkage), which depends on sparsity and on how the
signal energy concentrates in a few coefficients.
\end{enumerate}
For linear reconstruction from noisy coefficients, separability by itself does not provide an intrinsic
noise-reduction mechanism. The quantities that control noise amplification are the frame bounds (or,
equivalently, the conditioning of the frame operator) and the Bessel bound of the synthesis window used in the
reconstruction. Likewise, in the approximately dual setting, the key bias term is the operator deviation
$\|I-\Theta_{g,\gamma,\Lambda}\|$, which quantifies how close the mixed reconstruction operator is to the identity.
These constants are \emph{global} and are not determined solely by separability/non-separability.

The central feature of the present chirp-based approach is that it preserves the global stability constants and
the reconstruction error estimates inherited from the separable model.
More precisely, starting from a separable dual pair
$(\mathcal G(g,\Lambda),\mathcal G(\gamma,\Lambda))$ on a separable lattice, we apply a unitary chirp deformation
$U_C$ and the associated shear of the lattice. This use of a chirp should be understood as a concrete,
support-preserving instance of the broader metaplectic viewpoint in time-frequency analysis. In particular,
recent work on metaplectic Gabor frames shows that certain metaplectic time-frequency representations can be
represented, up to chirps and linear changes of variables, as rescaled short-time Fourier transforms
\citep{CorderoGiacchi24}. In contrast to that general program, our aim is deliberately narrower and constructive:
we restrict to the shear
$$
S_C(x,\omega)=(x,\omega+Cx)
$$
{and use it to transport explicit dual windows from product lattices to their chirp-sheared images.}

{This produces the chirped window
$$
g_C(x)=U_C g(x)=e^{\pi i x^T C x}\,g(x),
$$
supported and smooth whenever $g$ is, and defined on the sheared lattice $\Lambda_C=S_C\Lambda$.
When $g$ is a tensor-product window and the off-diagonal part of $C$ is nonzero, this
chirped window is genuinely non-separable in the sense of Definition~\ref{def:separable}.}
Because the deformation is unitary and intertwines time-frequency shifts in a precise covariance sense, the
frame bounds are preserved. Moreover, in the approximately dual regime, the reconstruction error is invariant:
$$
\|I-\Theta_{g_C,\gamma_C,\Lambda_C}\|
=
\|I-\Theta_{g,\gamma,\Lambda}\|,
\qquad \gamma_C:=U_C\gamma.
$$
As a consequence, the deterministic error bounds for noisy coefficients, and the corresponding finite-section
mean-square error decompositions, {carry over to the chirp-sheared setting \emph{without any deterioration of the constants}.} In short: in this paper non-separability is introduced to change the \emph{geometry} of the atoms,
not to trade stability for flexibility.

While the global stability and noise-propagation constants remain unchanged under our deformation, the
time-frequency \emph{geometry} can change substantially (the atoms become tilted/coupled).
{This provides a principled, fully explicit setting in which one can separate two effects:
the chirp changes a second-order phase-space geometry descriptor, while the frame stability
constants and the linear reconstruction bounds inherited from the original system remain unchanged.
We do not claim here that this covariance-level change, by itself, improves sparsity,
nonlinear approximation, denoising performance, conditioning, or computational complexity.
Such questions require additional signal-class-dependent analysis.}
The construction therefore isolates a clean separation of roles: separable windows supply explicit baselines
with controlled duals, and the chirp deformation generates non-separable higher-dimensional windows that inherit
the same reconstruction guarantees while potentially offering a better geometric match to non-axis-aligned
structures.

The main contributions of the paper may be summarized as follows. First, we recall and use a
tensorization principle showing how one-dimensional dual Gabor pairs generate dual pairs on separable
lattices in $L^2(\R^d)$. {Second, we record the chirp covariance principle in the form of a single unitary conjugacy
identity for mixed Gabor reconstruction operators. The preservation of frame bounds, exact
duality, canonical duals, and approximate-duality errors is then obtained as a direct consequence
of this identity.}
{Third, we obtain an explicit class of compactly supported smooth chirped tensor-product windows
on sheared lower block-triangular lattices of the form
$$
\Lambda_{A,B,C}
=
\left\{(Ak,CAk+B\ell): k,\ell\in\Z^d\right\}
=
\begin{pmatrix}
A&0\\
CA&B
\end{pmatrix}\Z^{2d},
\qquad C=C^T,
$$
where $A=\operatorname{diag}(a_1,\dots,a_d)$ and $B=\operatorname{diag}(b_1,\dots,b_d)$ are positive
diagonal matrices. Thus the time component is sampled on the rectangular lattice $A\Z^d$, while the
frequency component is sheared by the term $CAk$ before adding the rectangular frequency sampling
$B\Z^d$. When $C=0$, one recovers the product lattice $A\Z^d\times B\Z^d$; when $C\neq0$, the
phase-space shear is nontrivial. However, as noted below, this does not by itself imply that
$\Lambda_{A,B,C}$ is non-product as a set.}
{Fourth, we quantify the effect of the chirp deformation on phase-space covariance and show that
chirp-induced time--frequency coupling can be measured explicitly while the frame stability constants
and the linear reconstruction-error bounds remain unchanged.}

\subsection{Position of the present paper within the literature}

The present work is motivated by several complementary strands in the Gabor-frame literature.

The first strand is the constructive theory of Gabor frames and explicit window design. The constructive
philosophy goes back to the classical paper of Daubechies, Grossmann, and Meyer \citep{DGM86}, where explicit
and implementable nonorthogonal expansions were emphasized. For the general Hilbert-space frame background,
including frame operators and canonical duals, we refer to \citet{Christensen2016}. In higher dimension,
explicit smooth compactly supported windows for multivariate Gabor frames were constructed by Pfander,
Rashkov, and Wang \citep{PRW12}. Their geometric viewpoint is one of the main motivations for the present
paper. Our construction is different in scope: instead of building windows directly from fundamental domains,
we start from separable dual pairs and transport them through a unitary chirp deformation.

A closely related strand concerns Gabor analysis on non-separable time--frequency lattices
and, in the finite setting, on non-product subgroups of the time--frequency plane.
Feichtinger, Strohmer, and Christensen studied finite Gabor expansions indexed by
general additive subgroups rather than only product lattices, showing that the associated
dual frame can still be generated by a single dual atom \citep{FeichtingerStrohmerChristensen1995}.
Feichtinger and Kozek developed an operator-theoretic framework for TF-lattice-invariant
operators on elementary locally compact Abelian groups \citep{FeichtingerKozek1998}.
Related multidimensional non-separable Gabor expansions were considered in
\citep{FeichtingerKozekPrinzStrohmer1996}.

In the continuous and discrete signal-processing literature, van Leest and Bastiaans
formulated Gabor expansions and transforms on non-separable time--frequency lattices
by representing the lattice generator in Hermite normal form
\citep{VanLeestBastiaans2000}. In that representation, the passage from a rectangular
lattice to a sheared lattice is implemented by quadratic phase multiplications of the
signal, the synthesis and analysis windows, and the coefficient sequence. Discrete-time
implementations of such schemes were subsequently described through filter-bank,
shear, and Zak-transform viewpoints \citep{VanLeest1999,VanLeestBastiaans2004}.

The present paper is complementary to these works. We do not develop a discrete
implementation theory, nor do we derive Zak-transform product formulas or filter-bank
realizations. Instead, we work on $L^2(\mathbb R^d)$ and use one explicit unitary chirp
covariance identity to transport tensor-product dual Gabor pairs, including exact duality,
canonical duals, approximate-duality errors, compact support, and smoothness. Moreover,
we distinguish throughout between three different notions: a nontrivial phase-space shear,
a non-product lattice, and genuine functional non-separability of the window.

A second relevant strand is the symplectic and metaplectic viewpoint in time-frequency analysis.
The chirp factor
$$
e^{\pi i x^T Cx}
$$
is one of the elementary metaplectic building blocks, and the associated shear
$$
(x,\omega)\mapsto (x,\omega+Cx)
$$
is a simple symplectic transformation. Recent work on metaplectic Gabor frames provides a broad framework in
which Gabor-type systems are generated from metaplectic time-frequency representations \citep{CorderoGiacchi24}.
The present paper may be viewed as a concrete and explicitly computable subcase of this general philosophy:
we do not develop a full metaplectic theory, but we exploit one chirp shear to obtain non-separable windows
while preserving compact support, smoothness, frame bounds, and reconstruction errors.

A third strand concerns the control of dual windows. Exact compactly supported dual windows were studied in
detail by Stoeva \citep{Sto22}, while approximately dual Gabor systems and almost perfect reconstruction were
developed by Christensen, Janssen, Kim, and Kim \citep{CJKK18}. From a more computational perspective,
Werther, Eldar, and Subbanna \citep{WES05} studied dual Gabor frames different from the canonical dual and
emphasized the role of alternative duals in reducing computational complexity and improving conditioning.
Our approach is complementary to these works: we do not construct a new dual by solving an additional
optimization or inversion problem, but rather show that exact and approximate duality can be transported
unitarily from a separable model to a non-separable one.

A closely related finite-dimensional perspective has recently been developed by
\citet{AsipchukDeCarliRodriguez26}, who study finite Gabor systems whose frame-operator matrices are unitarily equivalent, through explicit and computationally efficient transformations, to block-diagonal matrices. In particular, they show that certain finite Gabor systems indexed by Cartesian products of subsets of $\mathbb Z_N$ admit such a block structure, which makes the inversion of the frame operator substantially more tractable. Although our setting is the continuous Hilbert space $L^2(\mathbb R^d)$, their results provide a useful finite-dimensional counterpart to the computational question of how frame-operator structure can simplify reconstruction. In contrast, the present paper does not diagonalize the frame operator directly; instead, it uses chirp covariance to transport duality and frame-operator properties from a separable lattice to a sheared, non-separable one.

A further relevant perspective is provided by multi-window and Banach-space theories. The work of Balan,
Christensen, Krishtal, Okoudjou, and Romero \citep{BCKOR13} establishes stability and reconstruction results
for multi-window Gabor frames in Wiener amalgam spaces, showing that operator-algebra methods preserve
localization properties of dual systems. Rational and mixed rational subspace multi-window Gabor frames were
studied in \citep{ZL14,ZL18}, while Gabor frame sets for subspaces were analyzed in \citep{LL11}. These works
are not used directly in the proofs below, but they indicate a broader framework in which explicit dual
constructions and stability of reconstruction remain highly relevant.

Finally, for Gaussian windows in several variables, Gröchenig \citep{Gro11} showed that the multivariate frame
problem is closely related to sampling of entire functions in Bargmann--Fock spaces of several complex variables.
This Gaussian program is conceptually different from the present compactly supported constructive approach.
For the optimization of Gaussian frame bounds on separable lattices, see also \citep{FS17}. A natural long-term
direction is to understand whether the chirp-transport philosophy developed here can interact with Gaussian
and Bargmann--Fock methods, {especially in the study of sheared or non-product lattices and lattice-shape optimization.}

\section{Preliminaries and notation}

Throughout the paper, for $x,\omega \in \R^d$ we denote by
$$
T_x f(t):=f(t-x),
\qquad
M_\omega f(t):=e^{2\pi i \omega\cdot t}f(t),
$$
and we write
$$
\pi(x,\omega):=M_\omega T_x.
$$
Given a lattice $\Lambda\subset \R^{2d}$ and a window $g\in L^2(\R^d)$, we consider the Gabor system
$$
\mathcal G(g,\Lambda):=\{\pi(\lambda)g:\lambda\in\Lambda\}.
$$
We say that $\mathcal G(g,\Lambda)$ is a frame for $L^2(\R^d)$ if there exist constants
$0<A\leq B<\infty$ such that
$$
A\|f\|_2^2
\leq
\sum_{\lambda\in\Lambda} |\inner{f,\pi(\lambda)g}|^2
\leq
B\|f\|_2^2,
\qquad \forall f\in L^2(\R^d).
$$
If $\mathcal G(g,\Lambda)$ is a frame, its frame operator is
$$
S_{g,\Lambda}f
:=
\sum_{\lambda\in\Lambda}\inner{f,\pi(\lambda)g}\,\pi(\lambda)g,
\qquad f\in L^2(\R^d)\, .
$$
Since $\mathcal{G}(g, \Lambda)$ is a frame, $S_{g, \Lambda}$ is bounded, positive, self-adjoint and invertible \citep{Christensen2016}. The canonical dual window is $\tilde{g}:=S_{g, \Lambda}^{-1} g$.

\begin{definition}
Let $g,\gamma\in L^2(\R^d)$ and let $\Lambda\subset \R^{2d}$ be a lattice.
We say that $\mathcal G(\gamma,\Lambda)$ is a \emph{dual Gabor frame} of
$\mathcal G(g,\Lambda)$ if
$$
f
=
\sum_{\lambda\in\Lambda}\inner{f,\pi(\lambda)\gamma}\,\pi(\lambda)g
=
\sum_{\lambda\in\Lambda}\inner{f,\pi(\lambda)g}\,\pi(\lambda)\gamma,
\qquad \forall f\in L^2(\R^d),
$$
with convergence in $L^2(\R^d)$.
\end{definition}

\begin{definition}
Let $g,\gamma\in L^2(\R^d)$ and let $\Lambda\subset \R^{2d}$ be a lattice.
We define the mixed reconstruction operator
$$
\Theta_{g,\gamma,\Lambda}f
:=
\sum_{\lambda\in\Lambda}\inner{f,\pi(\lambda)g}\,\pi(\lambda)\gamma.
$$
The pair $(\mathcal G(g,\Lambda),\mathcal G(\gamma,\Lambda))$ is called an
\emph{approximately dual pair} if
$$
\|I-\Theta_{g,\gamma,\Lambda}\|<1\, ,
$$
where $\|\cdot\|$ is the operator norm on $L^2(\R^d)$.
\end{definition}
\begin{remark}[Operator norm and reconstruction error]\label{rem:operator_norm_error}
For an approximately dual pair $(\mathcal G(g,\Lambda),\mathcal G(\gamma,\Lambda))$, the quantity
$$\|I-\Theta_{g,\gamma,\Lambda}\|$$
will be referred to as the \emph{reconstruction error} (or \emph{approximation error}), since it controls the
deviation of the mixed reconstruction operator $\Theta_{g,\gamma,\Lambda}$ from the identity.
\end{remark}

{The class of sheared lower block-triangular lattices considered in this paper is generated by matrices of the following form. Let}
$$
A=\operatorname{diag}(a_1,\dots,a_d),
\qquad
B=\operatorname{diag}(b_1,\dots,b_d),
\qquad a_j,b_j>0,
$$
and let $C\in \R^{d\times d}$ be symmetric. We define
$$
\Lambda_{A,B,C}
:=
\left\{
(Ak,\,CAk+B\ell):k,\ell\in \Z^d
\right\}
=
\begin{pmatrix}
A & 0\\
CA & B
\end{pmatrix}\Z^{2d}.
$$
When $C=0$, the lattice is separable:
$$
\Lambda_{A,B,0}=A\Z^d\times B\Z^d.
$$
\begin{remark}[Shear, product structure, and non-separability]
{The condition $C\neq 0$ means that the phase-space shear
$$
S_C(x,\omega)=(x,\omega+Cx)
$$
is nontrivial. This should not be confused with the assertion that the lattice
$\Lambda_{A,B,C}$ is non-product as a set, nor with the assertion that the chirped
window is non-separable as a function.}

{Indeed, the sheared lattice
$$
\Lambda_{A,B,C}
=
\{(Ak,CAk+B\ell):k,\ell\in\mathbb Z^d\}
$$
coincides with the product lattice $A\mathbb Z^d\times B\mathbb Z^d$ whenever
$$
CA=BN
$$
for some integer matrix $N\in M_d(\mathbb Z)$. In that case
$$
CAk+B\ell=B(Nk+\ell),
$$
and therefore
$$
\Lambda_{A,B,C}
=
\{(Ak,B(Nk+\ell)):k,\ell\in\mathbb Z^d\}
=
A\mathbb Z^d\times B\mathbb Z^d.
$$
Thus a nontrivial shear need not produce a non-product lattice; it may merely reparametrize
the original product lattice.}

{On the other hand, genuine non-separability of the chirped tensor-product window
$$
U_Cg(x)=e^{\pi i x^TCx}\prod_{m=1}^d g_m(x_m)
$$
is a different issue. It is caused by off-diagonal entries of $C$, not by the mere fact that
$C\neq0$. Thus the three notions--nontrivial chirp deformation, non-product lattice geometry,
and non-separability of the window--must be kept distinct.}
\end{remark}
This convention should be compared with the terminology used in parts of the signal-processing
literature, where a time--frequency lattice is often called non-separable when time shifts and
modulations are no longer independent operations \citep{VanLeestBastiaans2000,VanLeestBastiaans2004}.
In the present paper we use ``non-separable'' primarily for functions, and we explicitly separate
this functional property from the product or non-product nature of the underlying lattice.
\begin{definition}\label{def:separable}
A function $g:\R^d\to \C$ is called \emph{separable} if there exist one-variable functions
$g_1,\dots,g_d:\R\to\C$ such that such that
$$
g(x_1,\dots,x_d)=\prod_{j=1}^d g_j(x_j).
$$
Otherwise $g$ is called \emph{non-separable}.
\end{definition}

\section{Tensor-product dual pairs on separable lattices}

In this section we use the standard Hilbert tensor product
$$
\bigotimes_{j=1}^d L^2(\R),
$$
defined as the completion of finite linear combinations of simple tensors with respect to the inner product
$$
\left\langle \bigotimes_{j=1}^d f_j,\bigotimes_{j=1}^d h_j \right\rangle
:=
\prod_{j=1}^d \langle f_j,h_j\rangle_{L^2(\R)}.
$$
Thus, for a simple tensor one has
$$
\left\| \bigotimes_{j=1}^d f_j \right\|^2_{\bigotimes_{j=1}^d L^2(\R)}
=
\prod_{j=1}^d \|f_j\|_{L^2(\R)}^2.
$$
We shall use the canonical unitary identification
$$
\bigotimes_{j=1}^d L^2(\R) \cong L^2(\R^d),
$$
under which
$$
f_1\otimes\cdots\otimes f_d
\quad\longleftrightarrow\quad
\left[(x_1,\dots,x_d)\mapsto \prod_{j=1}^d f_j(x_j)\right].
$$
Indeed, by Fubini's theorem,
$$
\left\|\prod_{j=1}^d f_j(x_j)\right\|_{L^2(\R^d)}^2
=
\int_{\R^d}\prod_{j=1}^d |f_j(x_j)|^2\,dx_1\cdots dx_d
=
\prod_{j=1}^d\|f_j\|_{L^2(\R)}^2.
$$
By density, this identification extends from simple tensors to the completed Hilbert tensor product.
Throughout the paper we use this identification without further comment.

We first record the elementary tensorization principle that will serve as the starting point
for our constructions.

\begin{proposition}\label{prop:tensor_duals}
For each $j=1,\dots,d$, let $a_j,b_j>0$ and let $g_j,\gamma_j\in L^2(\R)$ be such that
$$
\mathcal G(g_j,a_j\Z\times b_j\Z)
\quad\text{and}\quad
\mathcal G(\gamma_j,a_j\Z\times b_j\Z)
$$
are dual Gabor frames for $L^2(\R)$.

{Define
$$
g:=g_1\otimes \cdots \otimes g_d,
\qquad
\gamma:=\gamma_1\otimes \cdots \otimes \gamma_d,
$$
and identify these tensors with functions on $\R^d$ by
$$
g(x_1,\dots,x_d)=\prod_{j=1}^d g_j(x_j),
\qquad
\gamma(x_1,\dots,x_d)=\prod_{j=1}^d \gamma_j(x_j).
$$
Then $\mathcal G(g,\Lambda_{A,B,0})$ and $\mathcal G(\gamma,\Lambda_{A,B,0})$
are dual Gabor frames for $L^2(\R^d)$.}

If, in addition, each $g_j,\gamma_j$ belongs to $C_c^\infty(\R)$, then
$g,\gamma\in C_c^\infty(\R^d)$.
\end{proposition}

\begin{proof}
For $\lambda=(Ak,B\ell)\in \Lambda_{A,B,0}$, one has
$$
\pi(\lambda)g
=
\bigotimes_{j=1}^d \big(M_{\ell_j b_j}T_{k_j a_j}g_j\big),
\qquad
\pi(\lambda)\gamma
=
\bigotimes_{j=1}^d \big(M_{\ell_j b_j}T_{k_j a_j}\gamma_j\big).
$$
Let $f=f_1\otimes\cdots\otimes f_d$ be a simple tensor. Then
$$
\inner{f,\pi(Ak,B\ell)\gamma}
=
\prod_{j=1}^d
\inner{f_j,M_{\ell_j b_j}T_{k_j a_j}\gamma_j}.
$$
Hence, for rectangular partial sums,
$$
\sum_{\substack{|k_j|\leq N\\ |\ell_j|\leq N}}
\inner{f,\pi(Ak,B\ell)\gamma}\,\pi(Ak,B\ell)g
=
\bigotimes_{j=1}^d
\left(
\sum_{\substack{|m|\leq N\\ |n|\leq N}}
\inner{f_j,M_{nb_j}T_{ma_j}\gamma_j}\,M_{nb_j}T_{ma_j}g_j
\right).
$$
Passing to the limit as $N\to\infty$ and using the one-dimensional duality for each factor, we obtain
$$
\sum_{\lambda\in\Lambda_{A,B,0}}
\inner{f,\pi(\lambda)\gamma}\,\pi(\lambda)g
=
f_1\otimes\cdots\otimes f_d
=
f.
$$
{The same argument yields the reverse reconstruction formula on simple tensors, and hence by linearity
on finite linear combinations of simple tensors.}

{Before extending the reconstruction identities by density, we verify that the tensorized systems are
Bessel. Let $B_j^g$ and $B_j^\gamma$ be Bessel bounds for
$\mathcal G(g_j,a_j\Z\times b_j\Z)$ and $\mathcal G(\gamma_j,a_j\Z\times b_j\Z)$, respectively. Define the
one-dimensional analysis operators
$$
\mathcal C_j^g f_j
:=
\big(\inner{f_j,M_{nb_j}T_{ma_j}g_j}\big)_{(m,n)\in\Z^2},
\qquad
\mathcal C_j^\gamma f_j
:=
\big(\inner{f_j,M_{nb_j}T_{ma_j}\gamma_j}\big)_{(m,n)\in\Z^2}.
$$
Then
$$
\|\mathcal C_j^g f_j\|_{\ell^2(\Z^2)}^2\le B_j^g\|f_j\|_2^2,
\qquad
\|\mathcal C_j^\gamma f_j\|_{\ell^2(\Z^2)}^2\le B_j^\gamma\|f_j\|_2^2.
$$
Consequently the Hilbert-space tensor product
$$
\mathcal C^g:=\bigotimes_{j=1}^d \mathcal C_j^g:
\bigotimes_{j=1}^d L^2(\R)\longrightarrow
\bigotimes_{j=1}^d \ell^2(\Z^2)\cong \ell^2(\Z^{2d})
$$
extends to a bounded operator with
$$
\|\mathcal C^g\|\le \prod_{j=1}^d \sqrt{B_j^g}.
$$
For simple tensors, $\mathcal C^g f$ is precisely the coefficient sequence
$$
\big(\inner{f,\pi(Ak,B\ell)g}\big)_{(k,\ell)\in\Z^d\times\Z^d}.
$$
By density and continuity this remains true for all $f\in L^2(\R^d)$, and therefore
$$
\sum_{k,\ell}\big|\inner{f,\pi(Ak,B\ell)g}\big|^2
\le
\left(\prod_{j=1}^d B_j^g\right)\|f\|_2^2,
\qquad f\in L^2(\R^d).
$$
Thus $\mathcal G(g,\Lambda_{A,B,0})$ is Bessel with upper bound
$B_g:=\prod_{j=1}^d B_j^g$. The same argument gives that
$\mathcal G(\gamma,\Lambda_{A,B,0})$ is Bessel with upper bound
$B_\gamma:=\prod_{j=1}^d B_j^\gamma$.}

Since finite linear combinations of simple tensors are dense in $L^2(\R^d)$, it remains to justify
that the reconstruction operators are bounded on $L^2(\R^d)$, so that the identities extend by continuity.
Let
$$
\Theta_{g,\gamma,\Lambda_{A,B,0}}f
:=
\sum_{\lambda\in\Lambda_{A,B,0}}\inner{f,\pi(\lambda)\gamma}\,\pi(\lambda)g,
\qquad f\in L^2(\R^d),
$$
and define analogously $\Theta_{\gamma,g,\Lambda_{A,B,0}}$ with the roles of $g$ and $\gamma$ interchanged.
{Let $B_g$ and $B_\gamma$ be Bessel bounds for
$G(g,\Lambda_{A,B,0})$ and $G(\gamma,\Lambda_{A,B,0})$},
respectively
$$
\sum_{\lambda\in\Lambda_{A,B,0}}|\inner{f,\pi(\lambda)g}|^2\le B_g\|f\|_2^2,
\qquad
\sum_{\lambda\in\Lambda_{A,B,0}}|\inner{f,\pi(\lambda)\gamma}|^2\le B_\gamma\|f\|_2^2,
\qquad \forall f\in L^2(\R^d).
$$
Fix $f,h\in L^2(\R^d)$. Then, by Cauchy--Schwarz in $\ell^2(\Lambda_{A,B,0})$,
\begin{align*}
\big|\inner{\Theta_{g,\gamma,\Lambda_{A,B,0}}f,h}\big|
&=
\left|\sum_{\lambda\in\Lambda_{A,B,0}}\inner{f,\pi(\lambda)\gamma}\,\inner{\pi(\lambda)g,h}\right|\\
&\le
\left(\sum_{\lambda\in\Lambda_{A,B,0}}|\inner{f,\pi(\lambda)\gamma}|^2\right)^{1/2}
\left(\sum_{\lambda\in\Lambda_{A,B,0}}|\inner{h,\pi(\lambda)g}|^2\right)^{1/2}\\
&\le
\sqrt{B_\gamma}\,\|f\|_2\;\sqrt{B_g}\,\|h\|_2.
\end{align*}
Taking the supremum over $\|h\|_2=1$ yields
$$
\|\Theta_{g,\gamma,\Lambda_{A,B,0}}f\|_2\le \sqrt{B_gB_\gamma}\,\|f\|_2,
\qquad \forall f\in L^2(\R^d),
$$
hence $\Theta_{g,\gamma,\Lambda_{A,B,0}}$ is bounded on $L^2(\R^d)$. The same argument shows that
$\Theta_{\gamma,g,\Lambda_{A,B,0}}$ is bounded. We have proved on simple tensors that $\Theta_{g,\gamma,\Lambda_{A,B,0}}f=f$ and
$\Theta_{\gamma,g,\Lambda_{A,B,0}}f=f$. By density of finite linear combinations of simple tensors and
boundedness of the reconstruction operators, both identities extend to all $f\in L^2(\R^d)$.

The regularity statement is immediate because finite tensor products of compactly supported
smooth functions are again compactly supported and smooth.
\end{proof}

\begin{remark}
Proposition \ref{prop:tensor_duals} is elementary, but it is important for the present paper because it provides the separable starting point from which non-separable constructions will be generated. In concrete applications, one may choose the one-dimensional dual pairs from compactly supported exact constructions as in \citep{Sto22}, or from other explicit one-dimensional Gabor-frame constructions in the spirit of \citep{DGM86}.
\end{remark}

\section{Chirp deformations and sheared lattices}
\label{sec:chirp}
The use of chirp and metaplectic deformations is consistent with the symplectic-geometric viewpoint that already plays an important role in higher-dimensional Gabor analysis; compare in particular \citep{PRW12}. In the present paper we restrict ourselves to a concrete chirp model because it allows a fully explicit transfer of support, smoothness, duality, and reconstruction estimates.

{We now pass from product lattices to chirp-sheared lattices through a simple metaplectic model,
namely chirp multiplication.}

\begin{definition}
Let $C\in\R^{d\times d}$ be symmetric. The associated chirp operator
$U_C:L^2(\R^d)\to L^2(\R^d)$ is defined by
$$
(U_C f)(t):=e^{\pi i\, t^T C t}f(t).
$$
\end{definition}

Since $|e^{\pi i\, t^T C t}|=1$, the operator $U_C$ is unitary on $L^2(\R^d)$.
Moreover, $U_C$ preserves compact support, smoothness, and the Schwartz class.

The following covariance relation is the continuous Hilbert-space counterpart of the
quadratic-phase shear mechanisms used in non-separable Gabor schemes on general
time--frequency lattices; compare, in particular, the Hermite-normal-form reduction in
\citet{VanLeestBastiaans2000}. Our use of the identity below is different in emphasis:
it is not used to derive an implementation from a rectangular algorithm, but to transport
frame-theoretic properties and explicit dual windows by unitary conjugacy.
\begin{lemma}[Covariance under chirp]\label{lem:chirp_covariance}
Let $C\in \R^{d\times d}$ be symmetric, and define
$$
S_C(x,\omega):=(x,\omega+Cx), \qquad (x,\omega)\in \R^{2d}.
$$
Then, for all $x,\omega\in \R^d$,
$$
\pi(S_C(x,\omega))\,U_C
=
e^{\pi i x^T C x}\,U_C\,\pi(x,\omega).
$$
\end{lemma}

\begin{proof}
Let $f\in L^2(\R^d)$ and $t\in \R^d$. Then
\begin{align*}
\big(\pi(S_C(x,\omega))U_C f\big)(t)
&=
e^{2\pi i (\omega+Cx)\cdot t}\,
(U_C f)(t-x)\\
&=
e^{2\pi i (\omega+Cx)\cdot t}\,
e^{\pi i (t-x)^T C (t-x)}f(t-x).
\end{align*}
Because $C$ is symmetric,
$$
(t-x)^TC(t-x)=t^TCt-2x^TCt+x^TCx.
$$
Therefore
\begin{align*}
\big(\pi(S_C(x,\omega))U_C f\big)(t)
&=
e^{2\pi i\omega\cdot t}
e^{2\pi i x^TCt}
e^{\pi i t^TCt}
e^{-2\pi i x^TCt}
e^{\pi i x^TCx}
f(t-x)\\
&=
e^{\pi i x^TCx}
e^{\pi i t^TCt}
e^{2\pi i\omega\cdot t}
f(t-x)\\
&=
e^{\pi i x^TCx}\,(U_C\pi(x,\omega)f)(t).
\end{align*}
This proves the claim.
\end{proof}

{\begin{proposition}[Unitary conjugacy of mixed Gabor reconstruction]\label{prop:mixed_conjugacy}
Let $\Lambda\subset\R^{2d}$ be a lattice, let $C=C^T$, and assume that
$\mathcal G(g,\Lambda)$ and $\mathcal G(\gamma,\Lambda)$ are Bessel systems. Then
$$
\Theta_{U_C g,U_C\gamma,S_C\Lambda}
=
U_C\,\Theta_{g,\gamma,\Lambda}\,U_C^{-1}.
$$
\end{proposition}
\begin{proof}
Let $f\in L^2(\R^d)$. By Lemma~\ref{lem:chirp_covariance}, for
$\lambda=(x,\omega)\in\Lambda$,
$$
\inner{f,\pi(S_C\lambda)U_C g}
=
e^{-\pi i x^TCx}
\inner{U_C^{-1}f,\pi(\lambda)g},
$$
and
$$
\pi(S_C\lambda)U_C\gamma
=
e^{\pi i x^TCx}U_C\pi(\lambda)\gamma.
$$
The phase factors cancel in the mixed reconstruction sum. Hence
$$
\begin{aligned}
\Theta_{U_C g,U_C\gamma,S_C\Lambda}f
&=
\sum_{\lambda\in\Lambda}
\inner{f,\pi(S_C\lambda)U_C g}\,\pi(S_C\lambda)U_C\gamma  \\
&=
U_C
\sum_{\lambda\in\Lambda}
\inner{U_C^{-1}f,\pi(\lambda)g}\,\pi(\lambda)\gamma  \\
&=
U_C\,\Theta_{g,\gamma,\Lambda}\,U_C^{-1}f.
\end{aligned}
$$
This proves the identity.
\end{proof}
\begin{remark}[Scope of the transport results]
The preceding identity is an elementary instance of metaplectic covariance. Accordingly, the
preservation of frame bounds, exact duality, canonical duals, and approximate-duality errors below
should be understood as consequences of this unitary equivalence, not as independent phenomena.
The constructive point of the present paper is to combine this transport principle with tensor-product
dual pairs in order to obtain explicit compactly supported chirped windows, and to identify when
the resulting windows are genuinely non-separable as functions.
\end{remark}}

\begin{corollary}[Transport of frames and duality]\label{cor:transport_duality}
Let $\Lambda\subset \R^{2d}$ be a lattice, let $C$ be symmetric, and put
$$
S_C\Lambda:=\{S_C\lambda:\lambda\in\Lambda\}.
$$
If $\mathcal G(g,\Lambda)$ is a Gabor frame with frame bounds $A,B$, then
$\mathcal G(U_C g,S_C\Lambda)$ is a Gabor frame with the same bounds.

If $\mathcal G(\gamma,\Lambda)$ is a dual Gabor frame of $\mathcal G(g,\Lambda)$, then
$\mathcal G(U_C\gamma,S_C\Lambda)$ is a dual Gabor frame of $\mathcal G(U_C g,S_C\Lambda)$.
\end{corollary}
\begin{proof}
{The equality of frame bounds follows from the coefficient identity
$$
\inner{f,\pi(S_C\lambda)U_Cg}
=
e^{-\pi i x^TCx}\inner{U_C^{-1}f,\pi(\lambda)g},
\qquad \lambda=(x,\omega),
$$
together with the unitarity of $U_C$.}

{If $\mathcal G(\gamma,\Lambda)$ is dual to $\mathcal G(g,\Lambda)$, then the two mixed
reconstruction operators associated with $(g,\gamma)$ are the identity. By
Proposition~\ref{prop:mixed_conjugacy}, the corresponding transported mixed reconstruction
operators are
$$
U_C I U_C^{-1}=I.
$$
Hence $\mathcal G(U_C\gamma,S_C\Lambda)$ is dual to
$\mathcal G(U_Cg,S_C\Lambda)$.}
\end{proof}

{\begin{corollary}[Frame-operator and canonical-dual transport]\label{cor:frame_operator_conjugacy}
As a special case of Proposition~\ref{prop:mixed_conjugacy}, one has
$$
S_{U_C g,S_C\Lambda}
=
U_C\,S_{g,\Lambda}\,U_C^{-1}.
$$
Consequently,
$$
\widetilde{(U_C g)}
=
U_C\,\widetilde g.
$$
\end{corollary}}
{\begin{proof}
Apply Proposition~\ref{prop:mixed_conjugacy} with $\gamma=g$. The statement about the
canonical dual follows by inverting the conjugated frame operator.
\end{proof}}

\section{Explicit compactly supported dual pairs on chirp-sheared lattices}

{We now combine the tensor-product construction with the chirp transport principle.
This yields an explicit family of compactly supported smooth chirped windows on sheared
lower block-triangular lattices.}

\begin{theorem}\label{thm:main_constructive}
For each $j=1,\dots,d$, let $a_j,b_j>0$ and let $g_j,\gamma_j\in C_c^\infty(\R)$ be such that
$$
\mathcal G(g_j,a_j\Z\times b_j\Z)
\quad\text{and}\quad
\mathcal G(\gamma_j,a_j\Z\times b_j\Z)
$$
are dual Gabor frames for $L^2(\R)$.
Let $A=\operatorname{diag}(a_1,\dots,a_d)$, $B=\operatorname{diag}(b_1,\dots,b_d)$,
and let $C\in \R^{d\times d}$ be symmetric.

Define
$$
g(x):=e^{\pi i x^TCx}\prod_{j=1}^d g_j(x_j),
\qquad
\gamma(x):=e^{\pi i x^TCx}\prod_{j=1}^d \gamma_j(x_j),
\qquad x\in\R^d.
$$
Then the Gabor systems $\mathcal G(g,\Lambda_{A,B,C})$ and
$\mathcal G(\gamma,\Lambda_{A,B,C})$ are dual Gabor frames for $L^2(\R^d)$.
Moreover,
$$
g,\gamma\in C_c^\infty(\R^d).
$$
\end{theorem}

\begin{proof}
Set
$$
g^{(0)}:=g_1\otimes\cdots\otimes g_d,
\qquad
\gamma^{(0)}:=\gamma_1\otimes\cdots\otimes \gamma_d.
$$
By Proposition \ref{prop:tensor_duals}, the Gabor systems
$\mathcal G(g^{(0)},\Lambda_{A,B,0})$ and $\mathcal G(\gamma^{(0)},\Lambda_{A,B,0})$
are dual Gabor frames for $L^2(\R^d)$.

Next observe that
$$
\Lambda_{A,B,C}=S_C(\Lambda_{A,B,0}),
$$
because
$$
S_C(Ak,B\ell)=(Ak,\,CAk+B\ell).
$$
Also, by definition,
$$
U_C g^{(0)}(x)=e^{\pi i x^TCx}\prod_{j=1}^d g_j(x_j)=g(x),
$$
and similarly $U_C \gamma^{(0)}=\gamma$.

Applying Proposition \ref{cor:transport_duality} with $\Lambda=\Lambda_{A,B,0}$,
we conclude that $\mathcal G(g,\Lambda_{A,B,C})$ and
$\mathcal G(\gamma,\Lambda_{A,B,C})$ are dual Gabor frames.

Finally, multiplication by the smooth unimodular factor $e^{\pi i x^TCx}$ preserves
compact support and smoothness, hence $g,\gamma\in C_c^\infty(\R^d)$.
\end{proof}

{The previous theorem already yields a wide class of explicit higher-dimensional chirped windows.
The next proposition identifies the precise condition under which the chirp factor genuinely creates
non-separability of the window.}

\begin{proposition}\label{prop:nonseparable}
Assume that $C=(c_{jk})_{j,k=1}^d$ is symmetric and {the off-diagonal part of $C$ is nonzero, i.e.
there exist indices $j\neq k$ such that
$c_{jk}\neq0$.}
Assume moreover that each one-dimensional factor $g_m$ is nonzero on some nonempty open interval $I_m\subset \R$.
Then the function
$$
g(x)=e^{\pi i x^TCx}\prod_{m=1}^d g_m(x_m),
\qquad x=(x_1,\dots,x_d)\in \R^d,
$$
cannot be written in the form
$$
g(x)=\prod_{m=1}^d h_m(x_m)
$$
for one-variable functions $h_m$. In particular, $g$ is non-separable.
\end{proposition}
\begin{remark}[Why the off-diagonal assumption is necessary]
{The assumption that $C$ has a nonzero off-diagonal entry cannot be replaced by the weaker condition
$C\neq 0$. Indeed, if $C$ is diagonal, say
$$
C=\operatorname{diag}(c_1,\dots,c_d),
\qquad C\neq 0,
$$
then
$$
x^T Cx=\sum_{m=1}^d c_m x_m^2.
$$
Consequently,
$$
e^{\pi i x^T Cx}
=
\prod_{m=1}^d e^{\pi i c_m x_m^2}.
$$
Thus, for a separable tensor-product window
$$
g(x)=\prod_{m=1}^d g_m(x_m),
$$
one obtains
$$
U_C g(x)
=
e^{\pi i x^T Cx}\prod_{m=1}^d g_m(x_m)
=
\prod_{m=1}^d
\left(e^{\pi i c_m x_m^2}g_m(x_m)\right).
$$
Hence $U_C g$ is still separable. A diagonal chirp only modifies the individual one-dimensional
factors; it does not create any coupling between distinct coordinates.}

{Therefore the condition $C\neq 0$ is not sufficient to generate a genuinely non-separable window.
The relevant condition is precisely that the off-diagonal part of $C$ be nonzero, equivalently that
there exist indices $j\neq k$ such that $c_{jk}\neq 0$.}
\end{remark}
\begin{proof}
Suppose, by contradiction, that there exist one-variable functions $h_1,\dots,h_d$ such that
$$
g(x)=\prod_{m=1}^d h_m(x_m), \qquad x\in \R^d.
$$
Choose indices $j\neq k$ such that $c_{jk}\neq 0$.
By assumption, for each $m=1,\dots,d$ there exists a nonempty open interval $I_m\subset \R$ on which $g_m$ does not vanish.
Fix arbitrary points $\xi_m\in I_m$ for all $m\neq j,k$.

For $u\in I_j$ and $v\in I_k$, define
$$
x(u,v):=(x_1,\dots,x_d)\in \R^d
$$
by setting
$$
x_j=u,\qquad x_k=v,\qquad x_m=\xi_m \quad \text{for } m\neq j,k.
$$
Since $g$ is assumed separable, the two-variable function
$$
F(u,v):=g(x(u,v))
$$
must have the form
$$
F(u,v)=A\,h_j(u)h_k(v),
\qquad
A:=\prod_{m\neq j,k} h_m(\xi_m),
$$
and therefore it satisfies the multiplicative identity
$$
F(u_1,v_1)F(u_2,v_2)=F(u_1,v_2)F(u_2,v_1)
\qquad
\text{for all } u_1,u_2\in I_j,\ v_1,v_2\in I_k.
$$
We now compute the same quantity using the explicit formula for $g$.

Set
$$
G_0:=\prod_{m\neq j,k} g_m(\xi_m).
$$
Then, for $u\in I_j$ and $v\in I_k$,
$$
F(u,v)=e^{\pi i\, x(u,v)^T C x(u,v)}\,G_0\,g_j(u)g_k(v).
$$
Hence
\begin{align*}
&F(u_1,v_1)F(u_2,v_2)-F(u_1,v_2)F(u_2,v_1)\\
&=
G_0^2\,g_j(u_1)g_j(u_2)g_k(v_1)g_k(v_2)
\Big(
e^{\pi i\,[x(u_1,v_1)^TCx(u_1,v_1)+x(u_2,v_2)^TCx(u_2,v_2)]}\\
&\hspace{7em}
-
e^{\pi i\,[x(u_1,v_2)^TCx(u_1,v_2)+x(u_2,v_1)^TCx(u_2,v_1)]}
\Big).
\end{align*}
Since each $g_m$ is nonzero on $I_m$, the prefactor
$$
G_0^2\,g_j(u_1)g_j(u_2)g_k(v_1)g_k(v_2)
$$
is nonzero. Therefore the multiplicative identity would force
$$
e^{\pi i\,[x(u_1,v_1)^TCx(u_1,v_1)+x(u_2,v_2)^TCx(u_2,v_2)]}
=
e^{\pi i\,[x(u_1,v_2)^TCx(u_1,v_2)+x(u_2,v_1)^TCx(u_2,v_1)]}.
$$
Equivalently,
$$
e^{\pi i \Delta}=1,
$$
where
$$
\Delta
:=
x(u_1,v_1)^TCx(u_1,v_1)+x(u_2,v_2)^TCx(u_2,v_2)
-x(u_1,v_2)^TCx(u_1,v_2)-x(u_2,v_1)^TCx(u_2,v_1).
$$
Because
$$
x^TCx=\sum_{r=1}^d c_{rr}x_r^2+2\sum_{1\le r<s\le d} c_{rs}x_rx_s,
$$
all terms depending only on one coordinate cancel in the alternating sum defining $\Delta$, and so do all mixed terms except the $(j,k)$-term. Thus
$$
\Delta
=
2c_{jk}\big(u_1v_1+u_2v_2-u_1v_2-u_2v_1\big)
=
2c_{jk}(u_1-u_2)(v_1-v_2).
$$
Therefore separability would imply
$$
e^{2\pi i\,c_{jk}(u_1-u_2)(v_1-v_2)}=1
\qquad
\text{for all } u_1,u_2\in I_j,\ v_1,v_2\in I_k.
$$
We now choose $u_1,u_2\in I_j$ and $v_1,v_2\in I_k$ such that
$$
0<\big|c_{jk}(u_1-u_2)(v_1-v_2)\big|<1.
$$
This is possible because $I_j$ and $I_k$ are open intervals and $c_{jk}\neq 0$.
For such a choice, the real number $c_{jk}(u_1-u_2)(v_1-v_2)$ is nonzero and cannot be an integer, hence
$$
e^{2\pi i\,c_{jk}(u_1-u_2)(v_1-v_2)}\neq 1,
$$
a contradiction. Therefore $g$ cannot be written as a product of one-variable functions. In particular, $g$ is non-separable.
\end{proof}

{The preceding proposition is qualitative. We now record a quantitative version of the same
phenomenon. The point is that the off-diagonal coefficient $c_{jk}$ can be detected by a
four-point multiplicative obstruction which vanishes identically for separable functions.}
\begin{theorem}[Exact four-point obstruction generated by off-diagonal chirps]
\label{thm:four_point_obstruction}
{Let $d\ge 2$, let $C=(c_{mn})_{m,n=1}^d$ be symmetric, and let
$$
g_C(x)
=
e^{\pi i x^T Cx}\prod_{m=1}^d g_m(x_m),
\qquad x\in\R^d.
$$
Fix two distinct indices $j\neq k$. Assume that, for each $m$, the factor $g_m$
is nowhere zero on a nonempty open interval $I_m\subset\R$. Fix points
$\xi_m\in I_m$ for all $m\neq j,k$, and define the two-variable restriction
$$
F_\xi(u,v)
:=
g_C(x_\xi(u,v)),
$$
where $x_\xi(u,v)\in\R^d$ is given by
$$
(x_\xi(u,v))_j=u,\qquad
(x_\xi(u,v))_k=v,\qquad
(x_\xi(u,v))_m=\xi_m\quad (m\neq j,k).
$$
Then $F_\xi$ is nowhere zero on $I_j\times I_k$, and for every
$u_1,u_2\in I_j$, $v_1,v_2\in I_k$, one has
$$
\frac{
F_\xi(u_1,v_1)F_\xi(u_2,v_2)
}{
F_\xi(u_1,v_2)F_\xi(u_2,v_1)
}
=
e^{2\pi i\,c_{jk}(u_1-u_2)(v_1-v_2)}.
$$
In particular, the four-point quotient is identically equal to $1$ for every
separable function, whereas for the chirped tensor-product window it detects exactly
the off-diagonal coefficient $c_{jk}$.}

{Consequently, for every pair of bounded intervals
$J_j\subset I_j$, $J_k\subset I_k$ of positive length, the normalized rectangular defect
$$
\mathcal R_{jk}(g_C;J_j,J_k)^2
:=
\frac{1}{|J_j|^2|J_k|^2}
\int_{J_j^2}\int_{J_k^2}
\left|
\frac{
F_\xi(u_1,v_1)F_\xi(u_2,v_2)
}{
F_\xi(u_1,v_2)F_\xi(u_2,v_1)
}
-1
\right|^2
\,dv_1\,dv_2\,du_1\,du_2
$$
satisfies the exact formula
$$
\mathcal R_{jk}(g_C;J_j,J_k)^2
=
\frac{4}{|J_j|^2|J_k|^2}
\int_{J_j^2}\int_{J_k^2}
\sin^2\!\left(
\pi c_{jk}(u_1-u_2)(v_1-v_2)
\right)
\,dv_1\,dv_2\,du_1\,du_2.
$$
Hence
$$
\mathcal R_{jk}(g_C;J_j,J_k)=0
\quad\Longleftrightarrow\quad
c_{jk}=0.
$$
Moreover, if $L_j:=|J_j|$ and $L_k:=|J_k|$, then, as $c_{jk}\to0$,
$$
\mathcal R_{jk}(g_C;J_j,J_k)^2
=
\frac{\pi^2}{9}\,c_{jk}^2 L_j^2L_k^2
+
O\!\left(c_{jk}^4 L_j^4L_k^4\right).
$$}
\end{theorem}
\begin{proof}
{Since each $g_m$ is nowhere zero on $I_m$, the restricted function
$F_\xi$ is nowhere zero on $I_j\times I_k$. We now compute the four-point quotient.}

{Write $x=x_\xi(u,v)$. The quadratic form $x^TCx$, restricted to the variables
$(u,v)$, can be decomposed as
$$
x_\xi(u,v)^T C x_\xi(u,v)
=
c_{jj}u^2+c_{kk}v^2+2c_{jk}uv+\alpha u+\beta v+\rho,
$$
where $\alpha,\beta,\rho$ depend on $C$ and on the fixed coordinates
$\xi_m$, but not on $u$ and $v$ in any other way. Hence
$$
F_\xi(u,v)
=
e^{\pi i(c_{jj}u^2+c_{kk}v^2+2c_{jk}uv+\alpha u+\beta v+\rho)}
\,g_j(u)g_k(v)
\prod_{m\neq j,k}g_m(\xi_m).
$$
Consider now
$$
Q(u_1,u_2,v_1,v_2)
:=
\frac{
F_\xi(u_1,v_1)F_\xi(u_2,v_2)
}{
F_\xi(u_1,v_2)F_\xi(u_2,v_1)
}.
$$
All one-variable factors cancel:
the factors involving $g_j$, $g_k$, the fixed coordinates, the terms
$c_{jj}u^2$, $c_{kk}v^2$, $\alpha u$, $\beta v$, and $\rho$
appear equally in numerator and denominator. The only surviving term is the mixed
term $2c_{jk}uv$. Therefore
$$
\begin{aligned}
Q(u_1,u_2,v_1,v_2)
&=
e^{
\pi i\,2c_{jk}
\big(
u_1v_1+u_2v_2-u_1v_2-u_2v_1
\big)
}\\
&=
e^{
2\pi i\,c_{jk}(u_1-u_2)(v_1-v_2)
}.
\end{aligned}
$$
This proves the first assertion.}

{We also recall the elementary reason why this quotient is an obstruction to
separability. Let $H$ be any nowhere-zero separable function on a rectangle
$J_j\times J_k$, say
$$
H(u,v)=a(u)b(v),
\qquad u\in J_j,\ v\in J_k,
$$
for some one-variable functions $a$ and $b$. Then, for all
$u_1,u_2\in J_j$ and $v_1,v_2\in J_k$,
$$
\frac{
H(u_1,v_1)H(u_2,v_2)
}{
H(u_1,v_2)H(u_2,v_1)
}
=
\frac{
a(u_1)b(v_1)a(u_2)b(v_2)
}{
a(u_1)b(v_2)a(u_2)b(v_1)
}
=
1.
$$
Thus the quantity
$$
\frac{
H(u_1,v_1)H(u_2,v_2)
}{
H(u_1,v_2)H(u_2,v_1)
}
-1
$$
vanishes identically for separable functions. Applied to $H=F_\xi$, the formula
computed above shows that the obstruction is generated exactly by the mixed chirp
coefficient $c_{jk}$.}

{Using the identity
$$
|e^{i\theta}-1|^2=4\sin^2(\theta/2),
$$
with
$$
\theta=2\pi c_{jk}(u_1-u_2)(v_1-v_2),
$$
we obtain the exact formula
$$
\mathcal R_{jk}(g_C;J_j,J_k)^2
=
\frac{4}{|J_j|^2|J_k|^2}
\int_{J_j^2}\int_{J_k^2}
\sin^2\!\left(
\pi c_{jk}(u_1-u_2)(v_1-v_2)
\right)
\,dv_1\,dv_2\,du_1\,du_2.
$$
This quantity is zero if $c_{jk}=0$. Conversely, suppose it is zero. Since the integrand is
continuous and nonnegative, it must vanish identically on
$J_j^2\times J_k^2$. Hence
$$
\sin\!\left(
\pi c_{jk}(u_1-u_2)(v_1-v_2)
\right)=0
$$
for all $u_1,u_2\in J_j$ and $v_1,v_2\in J_k$. Because $J_j$ and $J_k$
have positive length, the product
$(u_1-u_2)(v_1-v_2)$ ranges over a nontrivial interval containing $0$.
The only way the above sine can vanish identically on such a continuum is $c_{jk}=0$.
Thus
$$
\mathcal R_{jk}(g_C;J_j,J_k)=0
\quad\Longleftrightarrow\quad
c_{jk}=0.
$$
Finally, for $c_{jk}\to0$, we use
$$
\sin^2 z=z^2+O(z^4).
$$
Therefore
$$
\begin{aligned}
\mathcal R_{jk}(g_C;J_j,J_k)^2
&=
\frac{4\pi^2c_{jk}^2}{|J_j|^2|J_k|^2}
\left(
\int_{J_j^2}(u_1-u_2)^2\,du_1\,du_2
\right)
\left(
\int_{J_k^2}(v_1-v_2)^2\,dv_1\,dv_2
\right)\\
&\qquad
+
O\!\left(c_{jk}^4 L_j^4L_k^4\right).
\end{aligned}
$$
For an interval $J$ of length $L$,
$$
\int_{J^2}(s-t)^2\,ds\,dt=\frac{L^4}{6}.
$$
Hence
$$
\mathcal R_{jk}(g_C;J_j,J_k)^2
=
\frac{4\pi^2c_{jk}^2}{L_j^2L_k^2}
\frac{L_j^4}{6}\frac{L_k^4}{6}
+
O\!\left(c_{jk}^4 L_j^4L_k^4\right),
$$
that is,
$$
\mathcal R_{jk}(g_C;J_j,J_k)^2
=
\frac{\pi^2}{9}\,c_{jk}^2L_j^2L_k^2
+
O\!\left(c_{jk}^4L_j^4L_k^4\right).
$$
This completes the proof.}
\end{proof}
\begin{remark}
Theorem \ref{thm:main_constructive} should be viewed as a constructive complement to the
geometric higher-dimensional window constructions available for special lattice classes.
Its strength is that once an exact dual pair is available on a separable lattice, the passage to the
sheared lattice $\Lambda_{A,B,C}$ is completely explicit and preserves compact support and smoothness.
\end{remark}

\section{Approximate duals and approximation of the canonical dual}
This section is motivated by two complementary ideas from the literature. First, exact compactly supported dual windows may be approximated in a controlled way, while remaining within the class of exact duals, as shown in \citep{Sto22}. Second, one may also work with approximately dual systems and quantify the resulting reconstruction error, as in \citep{CJKK18}. {The point of this section is not to introduce a new approximation mechanism, but to spell out
how existing exact-dual and approximate-dual constructions behave under the chirp transport
principle of Proposition~\ref{prop:mixed_conjugacy}.}

\begin{corollary}[Transport of approximate duals]\label{cor:approx_dual_transport}
Let $\Lambda\subset \R^{2d}$ be a lattice and let $C$ be symmetric.
If $(\mathcal G(g,\Lambda),\mathcal G(\gamma,\Lambda))$ is an approximately dual pair, then
$$
(\mathcal G(U_C g,S_C\Lambda),\mathcal G(U_C\gamma,S_C\Lambda))
$$
is also an approximately dual pair, with exactly the same reconstruction error:
$$
\|I-\Theta_{U_C g,U_C\gamma,S_C\Lambda}\|
=
\|I-\Theta_{g,\gamma,\Lambda}\|\, ,
$$
according to Remark \ref{rem:operator_norm_error}.
\end{corollary}
{\begin{proof}
By Proposition~\ref{prop:mixed_conjugacy},
$$
I-\Theta_{U_C g,U_C\gamma,S_C\Lambda}
=
U_C\,(I-\Theta_{g,\gamma,\Lambda})\,U_C^{-1}.
$$
Since $U_C$ is unitary, operator norm is preserved under conjugation.
\end{proof}}

The preceding corollary shows that any explicitly constructed approximate dual on a separable lattice can be transferred to the corresponding sheared lattice without deterioration of the reconstruction error.

We next state a canonical-dual approximation result in the exact-dual regime.

\begin{theorem}[Transport of exact-dual iterations]\label{thm:iteration_transport}
Let $\Lambda\subset \R^{2d}$ be a lattice, let $C$ be symmetric, and let
$g,\gamma^{(0)}\in L^2(\R^d)$ be such that $\mathcal G(\gamma^{(0)},\Lambda)$ is a dual Gabor frame of
$\mathcal G(g,\Lambda)$.
Assume that $\rho>0$ is such that
$$
\|I-\rho S_{g,\Lambda}\|<1.
$$
Define recursively
$$
\gamma^{(n+1)}
:=
\rho g + (I-\rho S_{g,\Lambda})\gamma^{(n)},
\qquad n\geq 0.
$$
Set
$$
g_C:=U_C g,
\qquad
\Gamma_C^{(n)}:=U_C\gamma^{(n)},
\qquad
\Lambda_C:=S_C\Lambda.
$$
Then, for every $n\geq 0$,
$$
\Gamma_C^{(n+1)}
=
\rho g_C + (I-\rho S_{g_C,\Lambda_C})\Gamma_C^{(n)}.
$$
Moreover, if each $\gamma^{(n)}$ is an exact dual window of $\mathcal G(g,\Lambda)$, then each
$\Gamma_C^{(n)}$ is an exact dual window of $\mathcal G(g_C,\Lambda_C)$.
Moreover,
$$
\gamma^{(n)}\to \widetilde g
\qquad\text{and}\qquad
\Gamma_C^{(n)}\to \widetilde{g_C}
\quad\text{in }L^2(\R^d).
$$
\end{theorem}
\begin{remark}[Neumann-series interpretation of the iteration]\label{rem:neumann_series}
The assumption $\|I-\rho S_{g,\Lambda}\|<1$ is a standard sufficient condition ensuring that
$S_{g,\Lambda}$ is invertible and that its inverse admits a Neumann-series representation. Indeed,
set $R:=I-\rho S_{g,\Lambda}$. Then $\|R\|<1$ implies that the series
$$
\sum_{n=0}^\infty R^n
$$
converges in operator norm and satisfies $(I-R)\left(\sum_{n=0}^\infty R^n\right)=I$.
Since $I-R=\rho S_{g,\Lambda}$, it follows that
$$
S_{g,\Lambda}^{-1}=\rho\sum_{n=0}^\infty (I-\rho S_{g,\Lambda})^n
\qquad\text{(convergence in operator norm).}
$$
In particular, this representation identifies the expected fixed point of the iteration as the canonical
dual window $\widetilde g=S_{g,\Lambda}^{-1}g$. The convergence of the iteration to this fixed point will be
proved explicitly in the proof below.
\end{remark}
\begin{proof}
By Corollary \ref{cor:frame_operator_conjugacy},
$$
S_{g_C,\Lambda_C}=U_C S_{g,\Lambda} U_C^{-1}.
$$
Hence
\begin{align*}
\rho g_C + (I-\rho S_{g_C,\Lambda_C})\Gamma_C^{(n)}
&=
\rho U_C g + \big(I-\rho U_C S_{g,\Lambda}U_C^{-1}\big)U_C\gamma^{(n)}\\
&=
U_C\big(\rho g + (I-\rho S_{g,\Lambda})\gamma^{(n)}\big)\\
&=
U_C\gamma^{(n+1)}\\
&=
\Gamma_C^{(n+1)}.
\end{align*}
This proves the recursive identity.

If each $\gamma^{(n)}$ is an exact dual of $g$ on $\Lambda$, then Proposition
\ref{cor:transport_duality} implies that each $U_C\gamma^{(n)}=\Gamma_C^{(n)}$
is an exact dual of $U_C g=g_C$ on $\Lambda_C$.

We show now that $\gamma^{(n)}$ converge in $L^2(\mathbb{R}^d)$ to $\tilde{g}$ and $\tilde{g} \in L^2(\mathbb{R}^d)$. Rather than invoking the Banach-Caccioppoli fixed point theorem, we exploit the linear structure of the recursion and write the iterates explicitly. This gives a direct proof of convergence to the canonical dual window through the Neumann series. We prove directly the both things.

First we observe that the fixed point for the recursively is $\tilde{g}$, indeed if the limit exist
\begin{equation}
\lim\gamma^{(n+1)}=\lim\gamma^{(n)}=\bar{\gamma}
\end{equation}
is the fixed point:
\begin{equation}
    \bar{\gamma}=\rho g+(I-\rho S_{g,\Lambda})\bar{\gamma}=\rho g+\bar{\gamma}-\rho S_{g,\Lambda}\bar{\gamma}
\end{equation}
that is equivalent to:
\begin{equation}
    \rho g-\rho S_{g,\Lambda}\bar{\gamma}=0 \iff S_{g,\Lambda}\bar{\gamma}=g
\end{equation}
We observe that the error $R_{g,\Lambda}:=I-\rho S_{g,\Lambda}$ for hypotesis have $\left\|R_{g,\Lambda}\right\|<1$, so as in Remark \ref{rem:neumann_series}:
\begin{equation}
    S_{g,\Lambda}^{-1}=\rho(I-R_{g,\Lambda})^{-1}=\rho\displaystyle\sum_{n=0}^{\infty}R_{g,\Lambda}^n.
\end{equation}
Since the inverse of $S_{g,\Lambda}$ exists and $\bar{\gamma}=S^{-1}_{g,\Lambda}g$, the fixed point is exactly the dual $\tilde{g}=S^{-1}_{g,\Lambda}g=\rho\displaystyle\sum_{n=0}^{\infty}R_{g,\Lambda}^ng$ with $R^{(0)}_{g,\Lambda}=I$.

Now we show the exact closed formula for the iteration if we start from a $\gamma^{(0)} \in L^2(\mathbb{R}^d)$. We define $Z_n(g):=\sum_{k=0}^nR_{g,\Lambda}^kg$ the partial sum of the Neumann series.
The closed formula is:
\begin{equation} \gamma^{(n)}=R^{n}_{g,\Lambda}\gamma^{(0)}+\rho Z_{n-1}(g)\; \;  \forall n \ge 1
\end{equation}
that tends to $\tilde{g}$ for $n \to \infty$ since it is the fixed point, as said previously. The closed formula follows by a straightforward induction on $n$ (equivalently, by solving first the
homogeneous recursion and then adding a particular solution through the usual variation-of-constants argument).
We now prove that $\gamma^{(n)}$ converges to $\widetilde g$ in $L^2(\R^d)$:
\begin{equation}
\begin{split}
    \left\|\gamma^{(n)}-\tilde{g}\right\|_{2}=\left\|R_{g,\Lambda}^{n}\gamma^{(0)}+\rho Z_{n-1}(g)- S_{g,\Lambda}^{-1}g\right\|_{2} \le \left\|R_{g,\Lambda}^{n}\gamma^{(0)}\right\|_{2}+ \rho \left\|\displaystyle\sum_{k=n}^{\infty}R_{g,\Lambda}^kg\right\|_{2 }\\\le \left\|R_{g,\Lambda}^{n}\right\|\left\|\gamma^{(0)}\right\|_{2}+ \rho \left\|\displaystyle\sum_{k=n}^{\infty}R_{g,\Lambda}^kg\right\|_{2 } \le \left\|R_{g,\Lambda}\right\|^n\left\|\gamma^{(0)}\right\|_{2}+ \rho \left\|\displaystyle\sum_{k=n}^{\infty}R_{g,\Lambda}^kg\right\|_{2 }
    \end{split}
\end{equation} and because $\left\|R_{g,\Lambda}\right\|<1$ then for $n  \to   \infty,  \; \left\|R_{g,\Lambda}\right\|^n \to 0 $, and because the tail of a convergent series go to 0 for $n \to \infty$, then:
\begin{equation}
    \left\|\gamma^{(n)}-\tilde{g}\right\|_{2}\le  \left\|R_{g,\Lambda}\right\|^n\left\|\gamma^{(0)}\right\|_{2} +  \rho \left\|\displaystyle\sum_{k=n}^{\infty}R_{g,\Lambda}^kg\right\|_{2 } \to 0
\end{equation} for $n \to \infty$ ($\gamma^{(0)} \in L^2(\mathbb{R}^d)$). So converge in $L^2(\mathbb{R}^d)$, then $\tilde{g} \in L^2(\mathbb{R}^d)$ via the completeness of Hilbert space.

Finally, since $\gamma^{(n)}\to \widetilde g$ in $L^2(\mathbb{R}^d)$, then by continuity of the unitary operator $U_C$,
$$
\Gamma_C^{(n)}=U_C\gamma^{(n)}\to U_C\widetilde g.
$$
The closed formula for $\Gamma^{(n)}_C$ when we fix a start function $\gamma^{(0)} \in L^2(\mathbb{R}^d)$ is:
\begin{equation}
    \Gamma^{(n)}_C=U_C\gamma^{(n)}=U_CR^{n}_{g,\Lambda}\gamma^{(0)}+\rho U_CZ_{n-1}(g),\; \;  \forall n \ge 1
\end{equation}
 and tend to $U_C\tilde{g}=U_C S^{-1}_{g,\Lambda}g $. By Corollary \ref{cor:frame_operator_conjugacy}, $U_C\widetilde g=\widetilde{g_C}$.
\end{proof}

\begin{corollary}\label{cor:compact_support_iteration}
Under the assumptions of Theorem \ref{thm:iteration_transport}, if each $\gamma^{(n)}$
is compactly supported, then each $\Gamma_C^{(n)}$ is compactly supported.
If each $\gamma^{(n)}$ belongs to $\mathcal S(\R^d)$, then each $\Gamma_C^{(n)}$ belongs to $\mathcal S(\R^d)$. Here, $\mathcal S(\R^d)$ denotes the Schwartz space.
\end{corollary}

\begin{proof}
This follows immediately from the fact that multiplication by the chirp
$e^{\pi i x^TCx}$ preserves compact support and the Schwartz class.
\end{proof}

\begin{remark}
{When the initial separable dual sequence is chosen from an iterative exact-dual approximation scheme
for the canonical dual, in the spirit of \citep{Sto22}, Theorem \ref{thm:iteration_transport} shows that
the entire approximation mechanism passes unchanged to the chirp-sheared lattice. Likewise,
Corollary \ref{cor:approx_dual_transport} shows that the approximate-duality framework of
\citep{CJKK18} is invariant under the same chirp deformation.}
\end{remark}

\section{Further structural remarks}

The results above isolate a class of non-separable higher-dimensional Gabor systems for which one can
retain three desirable features simultaneously:
\begin{enumerate}[label=\textup{(\roman*)}]
    \item explicit formulas for the generating window and a dual window;
    \item preservation of compact support and smoothness;
    \item direct transfer of exact and approximate reconstruction properties from a separable model.
\end{enumerate}

{From a geometric viewpoint, the map $S_C$ gives a nontrivial phase-space shear whenever
$C\neq0$. However, this does not necessarily imply that the image lattice is non-product as a set:
if $CA=BN$ for some integer matrix $N$, then
$$
\Lambda_{A,B,C}=A\mathbb Z^d\times B\mathbb Z^d.
$$
Thus the lattice geometry and the functional non-separability of the window are distinct issues.
The associated chirped tensor-product window becomes genuinely non-separable when $C$ has a
nonzero off-diagonal entry, under the non-vanishing assumptions of
Proposition~\ref{prop:nonseparable}. At the same time, the proofs remain transparent because the
deformation is implemented by a concrete unitary operator.}

It is also worth emphasizing that the present strategy is compatible with several possible sources for the
initial separable dual pair:
\begin{itemize}
    \item one-dimensional compactly supported exact dual windows;
    \item tensor products of such pairs;
    \item higher-dimensional separable windows obtained by geometric constructions on product lattices.
\end{itemize}
Thus the method is not tied to a single explicit family of windows.

Although the present paper is formulated entirely on $L^2(\R^d)$, the constructive control of dual windows obtained here is also compatible with broader stability questions in function spaces adapted to time-frequency analysis. In this direction, the Banach-space and operator-algebra approach of \citep{BCKOR13} suggests that explicit dual constructions may have consequences beyond the Hilbert-space level.

\section{Quantitative effects of chirp-induced non-separability:
time--frequency coupling, stability, and robustness}\label{sec:quantitative}

This section complements the constructive results above by providing three quantitative
viewpoints on chirp-generated non-separable windows:
{(i) a phase-space covariance formalism that makes chirp-induced time--frequency coupling
quantitatively visible;}
(ii) a precise comparison between atom-level localization and frame-level stability within
our transport class;
(iii) deterministic and statistical reconstruction error bounds under coefficient noise,
in both exact-dual and approximately-dual regimes.

Throughout, $d\ge 1$ is fixed. We work with the time--frequency shifts
$\pi(x,\omega)=M_\omega T_x$ and the chirp operator
$$
(U_C f)(t)=e^{\pi i\, t^T C t}f(t),
\qquad C=C^T\in\R^{d\times d},
$$
as in Section~\ref{sec:chirp}.
We denote by $S_C:\R^{2d}\to\R^{2d}$ the associated shear map
$$
S_C(x,\omega)=(x,\omega+Cx),
\qquad
S_C=\begin{pmatrix}I&0\\ C&I\end{pmatrix},
\qquad
S_C^{-1}=\begin{pmatrix}I&0\\ -C&I\end{pmatrix}.
$$

\subsection{A phase-space covariance matrix and a separability defect}\label{subsec:covariance}

\subsubsection{Short-time Fourier transform and the Moyal identity}
Let $f,\phi\in L^2(\R^d)$ with $\phi\neq 0$. We define the short-time Fourier transform (STFT)
$$
V_\phi f(x,\omega):=\inner{f,\pi(x,\omega)\phi}
=\int_{\R^d} f(t)\,\overline{\phi(t-x)}\,e^{-2\pi i \omega\cdot t}\,dt,
\qquad (x,\omega)\in\R^{2d}.
$$
We now prove the Moyal identity in the present normalization:
\begin{equation}\label{eq:moyal}
\int_{\R^{2d}} |V_\phi f(x,\omega)|^2\,dx\,d\omega
=
\|f\|_2^2\,\|\phi\|_2^2.
\end{equation}
\begin{proof}
{We first prove the identity for
$f,\varphi\in C_c^\infty(\mathbb R^d)$.}

{
For fixed $x\in\mathbb R^d$, the function
$$
t\mapsto f(t)\overline{\varphi(t-x)}
$$
belongs to $C_c^\infty(\mathbb R^d)$,
hence to $L^2(\mathbb R^d)$.
Therefore, by Plancherel's theorem,
$$
\int_{\mathbb R^d}
|V_\varphi f(x,\omega)|^2\,d\omega
=
\int_{\mathbb R^d}
|f(t)|^2|\varphi(t-x)|^2\,dt.
$$}

Integrating in $x$ and applying Fubini,
\begin{align*}
\int_{\R^{2d}} |V_\phi f(x,\omega)|^2\,dx\,d\omega
&=
\int_{\R^d}\int_{\R^d} |f(t)|^2\,|\phi(t-x)|^2\,dt\,dx\\
&=
\int_{\R^d} |f(t)|^2\left(\int_{\R^d}|\phi(t-x)|^2\,dx\right)dt\\
&=
\int_{\R^d}|f(t)|^2\,dt\;\int_{\R^d}|\phi(u)|^2\,du
=
\|f\|_2^2\|\phi\|_2^2,
\end{align*}
{This proves \eqref{eq:moyal} for
$f,\varphi\in C_c^\infty(\mathbb R^d)$.}

{
Since $C_c^\infty(\mathbb R^d)$ is dense in
$L^2(\mathbb R^d)$, the general case follows by
approximating $f$ and $\varphi$ in $L^2$ and passing
to the limit using the continuity of the STFT
as a bilinear map on $L^2(\mathbb R^d)\times
L^2(\mathbb R^d)$.}
\end{proof}

\subsubsection{A normalized phase-space energy distribution}
Assume $\|f\|_2=\|\phi\|_2=1$ (this is without loss of generality in this subsection, since the normalization
will be explicit). Define the nonnegative function
$$
p_{f,\phi}(x,\omega):=|V_\phi f(x,\omega)|^2.
$$
By \eqref{eq:moyal}, $p_{f,\phi}\in L^1(\R^{2d})$ and $\int_{\R^{2d}} p_{f,\phi}=1$; hence
$p_{f,\phi}$ is a probability density on phase space.

We define the associated \emph{phase-space mean} and \emph{phase-space covariance matrix}
whenever the second moments are finite:
$$
m_{f,\phi}:=\int_{\R^{2d}} z\,p_{f,\phi}(z)\,dz \in \R^{2d},
\qquad
\Sigma_{f,\phi}:=\int_{\R^{2d}} (z-m_{f,\phi})(z-m_{f,\phi})^T\,p_{f,\phi}(z)\,dz \in \R^{2d\times 2d}.
$$
(Here we write $z=(x,\omega)\in\R^{2d}$ and $dz=dx\,d\omega$.)

\paragraph{Finite second moments in our compactly supported smooth setting.}
In the present paper, $f$ and $\phi$ are typically in $C_c^\infty(\R^d)$.
In that case, $x\mapsto V_\phi f(x,\omega)$ has compact support uniformly in $\omega$:
indeed, $V_\phi f(x,\omega)$ can be nonzero only if $\operatorname{supp} f\cap(x+\operatorname{supp}\phi)\neq\varnothing$,
i.e.\ $x\in \operatorname{supp} f-\operatorname{supp}\phi$, a compact set. Hence all $x$-moments are finite.
Moreover, for each fixed $x$, the function $t\mapsto f(t)\overline{\phi(t-x)}$ is in $C_c^\infty(\R^d)$,
so its Fourier transform in $\omega$ decays faster than any polynomial; therefore all $\omega$-moments are finite.
Consequently $\Sigma_{f,\phi}$ is well-defined for our windows.

\subsubsection{Exact shear covariance of the STFT energy under chirp}
We now show that the chirp deformation transports the phase-space density by the shear $S_C$.

\begin{proposition}[STFT covariance under chirp]\label{prop:stft_chirp_cov}
Let $C=C^T$. For all $f,\phi\in L^2(\R^d)$ and all $(x,\omega)\in\R^{2d}$,
\begin{equation}\label{eq:stft_covariance}
V_{U_C\phi}(U_C f)(x,\omega)
=
e^{-\pi i x^T C x}\,V_\phi f\!\big(S^{-1}_C(x,\omega)\big)
=
e^{-\pi i x^T C x}\,V_\phi f(x,\omega-Cx).
\end{equation}
In particular,
\begin{equation}\label{eq:energy_covariance}
|V_{U_C\phi}(U_C f)(x,\omega)|^2
=
|V_\phi f(x,\omega-Cx)|^2
=
p_{f,\phi}(x,\omega-Cx).
\end{equation}
\end{proposition}

\begin{proof}
By definition of the STFT and unitarity of $U_C$,
$$
V_{U_C\phi}(U_C f)(x,\omega)
=
\inner{U_C f,\pi(x,\omega)U_C\phi}
=
\inner{f, U_C^{-1}\pi(x,\omega)U_C\,\phi}.
$$
From Lemma~\ref{lem:chirp_covariance} we have, for all $(x,\omega)$,
$$
\pi(S_C(x,\omega))\,U_C
=
e^{\pi i x^T C x}\,U_C\,\pi(x,\omega).
$$

Multiplying on the left by $U_C^{-1}$ and on the right by $U_C^{-1}$ yields
$$
U_C^{-1}\pi(S_C(x,\omega))U_C
=
e^{\pi i x^T C x}\,\pi(x,\omega).
$$
Replacing $(x,\omega)$ with $S_C^{-1}(x,\omega)$, we obtain
$$
U_C^{-1}\pi(x,\omega)U_C
=
e^{\pi i x^T C x}\,\pi(S_C^{-1}(x,\omega)).
$$
Since $S_C^{-1}(x,\omega)=(x,\omega-Cx)$, this reads
$$
U_C^{-1}\pi(x,\omega)U_C
=
e^{\pi i x^T C x}\,\pi(x,\omega-Cx).
$$
Substituting into the STFT expression gives
$$
V_{U_C\phi}(U_C f)(x,\omega)
=
e^{-\pi i x^T C x}\,V_\phi f(S^{-1}_C(x,\omega)),
$$
which is \eqref{eq:stft_covariance}. Taking moduli yields \eqref{eq:energy_covariance}.
\end{proof}

\subsubsection{How the phase-space covariance transforms, and a defect of separability}
Assume $\|f\|_2=\|\phi\|_2=1$ so that $p_{f,\phi}$ is a probability density.
By \eqref{eq:energy_covariance}, for the chirped pair $(U_C f,U_C\phi)$ we have
\begin{equation}\label{eq:density_pullback}
p_{U_C f,U_C\phi}(z)=p_{f,\phi}(S_C^{-1} z),\qquad z\in\R^{2d}.
\end{equation}
Since $\det S_C=1$, the change of variables $z'=S_C z$ has Jacobian $1$.

\begin{proposition}[Transformation rule for mean and covariance]\label{prop:mean_cov_transform}
Assume $m_{f,\phi}$ and $\Sigma_{f,\phi}$ are well-defined. Then $m_{U_C f,U_C\phi}$ and
$\Sigma_{U_C f,U_C\phi}$ are well-defined and satisfy
\begin{equation}\label{eq:mean_transform}
m_{U_C f,U_C\phi}=S_C m_{f,\phi},
\end{equation}

and
\begin{equation}\label{eq:cov_transform}
\Sigma_{U_C f,U_C\phi}
=
S_C\,\Sigma_{f,\phi}\,S_C^T.
\end{equation}
\end{proposition}

 \begin{proof}
  From $p_{U_C f,U_C\phi}(z)=p_{f,\phi}(S_C^{-1} z)$, since $\det S_C=1$, also $\det S_C^{-1}=1$, therefore
$$
m_{U_C f,U_C\phi}
=
\int_{\R^{2d}} z\,p_{f,\phi}(S_C^{-1}z)\,dz.
$$
Set $z'=S_C^{-1}z$, so that $z=S_C z'$ and $dz=dz'$. Then
$$
m_{U_C f,U_C\phi}
=
\int_{\R^{2d}} S_C z'\,p_{f,\phi}(z')\,dz'
=
S_C\int_{\R^{2d}} z'\,p_{f,\phi}(z')\,dz'
=
S_C m_{f,\phi},
$$
which proves \eqref{eq:mean_transform}.
For the covariance, write $m:=m_{f,\phi}$ and $m_C:=m_{U_C f,U_C\phi}=S_C m$.
Then
$$
\Sigma_{U_C f,U_C\phi}
=
\int (z-m_C)(z-m_C)^T\,p_{f,\phi}(S_C^{-1}z)\,dz.
$$
Apply the same change of variables $z'=S_C^{-1}z$, so $z=S_C z'$:
$$
z-m_C=S_C z'-S_C m=S_C(z'-m),
$$
hence
$$
(z-m_C)(z-m_C)^T
=
S_C (z'-m)(z'-m)^T S_C^T.
$$
Therefore,
$$
\Sigma_{U_C f,U_C\phi}
=
\int S_C (z'-m)(z'-m)^T S_C^T\,p_{f,\phi}(z')\,dz'
$$
$$
=
S_C\left(\int (z'-m)(z'-m)^T p_{f,\phi}(z')\,dz'\right)S_C^T
=
S_C\,\Sigma_{f,\phi}\,S_C^T,
$$
which proves \eqref{eq:cov_transform}.
\end{proof}

To highlight how chirp produces \emph{coupling}, we block-decompose the covariance matrix as
$$
\Sigma_{f,\phi}=
\begin{pmatrix}
\Sigma_{xx} & \Sigma_{x\omega}\\
\Sigma_{\omega x} & \Sigma_{\omega\omega}
\end{pmatrix},
\qquad
\Sigma_{xx},\Sigma_{x\omega},\Sigma_{\omega x},\Sigma_{\omega\omega}\in\R^{d\times d}.
$$
Using $S_C=\big(\!\begin{smallmatrix}I&0\\C&I\end{smallmatrix}\!\big)$,  a direct block-matrix multiplication yields:
\begin{equation}\label{eq:block_transform}
\Sigma_{U_C f,U_C\phi}=
\begin{pmatrix}
\Sigma_{xx} &
\Sigma_{x\omega}+\Sigma_{xx}C^T\\[1mm]
\Sigma_{\omega x}+C\Sigma_{xx} &
\Sigma_{\omega\omega}
+C\Sigma_{x\omega}
+\Sigma_{\omega x}C^T
+C\Sigma_{xx}C^T
\end{pmatrix}.
\end{equation}

{\begin{remark}[Scope of the covariance descriptor]
The quantity introduced below is a second-order phase-space descriptor associated with the
chosen pair $(f,\phi)$, where $\phi$ is the STFT reference window used to define the
spectrogram density. It is not an invariant of the Gabor system
$\mathcal G(f,\Lambda)$ alone, nor of the lattice alone. It also depends on the chosen
splitting of phase space into time and frequency coordinates. Accordingly, it should be
interpreted only as a covariance-level measure of time--frequency tilt, not as a direct measure
of sparsity, approximation quality, denoising performance, conditioning, or computational
complexity.
\end{remark}}

\begin{definition}[A time--frequency coupling defect]\label{def:coupling_defect}
Assume $\Sigma_{f,\phi}$ exists. We define the \emph{time--frequency coupling defect} of $(f,\phi)$ as
$$
\mathsf D(f,\phi):=\|\Sigma_{x\omega}\|_{\mathrm{F}},
$$
where $\|\cdot\|_{\mathrm{F}}$ is the Frobenius norm.
When $\mathsf D(f,\phi)=0$, the phase-space covariance has no $x$--$\omega$ cross-terms
(with respect to the chosen STFT window $\phi$), which is a natural covariance-level notion of
\emph{no time--frequency tilt}.
\end{definition}

\begin{corollary}[Chirp creates coupling from a decoupled baseline]\label{cor:coupling_created}
Assume $\Sigma_{f,\phi}$ exists and $\Sigma_{x\omega}=0$.
Then for every symmetric $C$,
$$
\mathsf D(U_C f,U_C\phi)=\|\Sigma_{xx}C^T\|_{\mathrm{F}}.
$$
In particular, if $\Sigma_{xx}$ is positive definite and $C\neq 0$, then $\mathsf D(U_C f,U_C\phi)>0$,
so the chirped pair exhibits genuine time--frequency coupling at the covariance level.
\end{corollary}

\begin{proof}
If $\Sigma_{x\omega}=0$, then by \eqref{eq:block_transform} the new cross-block is
$\Sigma_{x\omega}^{(C)}=\Sigma_{xx}C^T$. Taking Frobenius norms gives the claim.
If $\Sigma_{xx}$ is positive definite, then $\Sigma_{xx}C^T=0$ implies $C=0$, hence for $C\neq 0$
we have $\|\Sigma_{xx}C^T\|_{\mathrm{F}}>0$.
\end{proof}

\begin{remark}[Interpretation for separable windows]\label{rem:separable_interpretation}
{If $f$ and $\phi$ are separable (tensor products) and are centered/symmetric so that the cross-moments vanish,
then $\Sigma_{x\omega}=0$ is natural: the corresponding phase-space distribution has axes aligned with the coordinate
directions. Corollary~\ref{cor:coupling_created} shows that the chirp factor $e^{\pi i x^TCx}$ produces a \emph{tilt}
controlled explicitly by $C$, making time--frequency coupling measurable by $\mathsf D$.
This covariance-level coupling should not be identified with functional non-separability: diagonal chirps may
produce time--frequency tilt while preserving separability of tensor-product windows. Genuine functional
non-separability is guaranteed by the off-diagonal condition in Proposition~\ref{prop:nonseparable}.}
\end{remark}

\subsection{Localization versus frame stability: invariants and non-implications}\label{subsec:local_vs_stability}

\subsubsection{Frame bounds are invariant under chirp transport}
We rephrase the invariance established earlier in a form that will be used below.
Let $\Lambda\subset\R^{2d}$ be a lattice. Define the analysis and synthesis operators
$$
C_{g,\Lambda}:L^2(\R^d)\to \ell^2(\Lambda),
\qquad
C_{g,\Lambda}f:=\big(\inner{f,\pi(\lambda)g}\big)_{\lambda\in\Lambda},
$$
$$
D_{g,\Lambda}:\ell^2(\Lambda)\to L^2(\R^d),
\qquad
D_{g,\Lambda}c:=\sum_{\lambda\in\Lambda} c_\lambda\,\pi(\lambda)g,
$$
(where the synthesis sum converges unconditionally in $L^2$ when $c\in\ell^2$ and $g$ is Bessel).
Then the frame operator is $S_{g,\Lambda}=D_{g,\Lambda}C_{g,\Lambda}$ and the mixed reconstruction operator is
$\Theta_{g,\gamma,\Lambda}=D_{\gamma,\Lambda}C_{g,\Lambda}$.

\begin{proposition}[Conjugacy of analysis/synthesis and stability constants]\label{prop:analysis_synthesis_conjugacy}
Let $C=C^T$ and $\Lambda_C:=S_C\Lambda$. Then
$$
\|C_{U_C g,\Lambda_C}\|=\|C_{g,\Lambda}\|,
\qquad
\|D_{U_C g,\Lambda_C}\|=\|D_{g,\Lambda}\|.
$$
In particular, the optimal Bessel bound of $\mathcal G(g,\Lambda)$ equals that of $\mathcal G(U_C g,\Lambda_C)$.
Moreover, if $\mathcal G(g,\Lambda)$ is a frame with bounds $A,B$, then $\mathcal G(U_C g,\Lambda_C)$ is a frame with the same bounds.
\end{proposition}

\begin{proof}
Fix $f\in L^2(\R^d)$ and $\lambda=(x,\omega)\in\Lambda$. By Lemma~\ref{lem:chirp_covariance},
$$
\pi(S_C\lambda)\,U_C
=
e^{\pi i x^TCx}\,U_C\,\pi(\lambda).
$$
Taking adjoints and using unitarity gives
$$
U_C^{-1}\pi(S_C\lambda)^*
=
e^{-\pi i x^TCx}\,\pi(\lambda)^*U_C^{-1}.
$$
Hence,
\begin{align*}
\inner{f,\pi(S_C\lambda)U_C g}
&=\inner{\pi(S_C\lambda)^*f, U_C g}
=\inner{U_C^{-1}\pi(S_C\lambda)^*f, g}\\
&=
e^{-\pi i x^TCx}\,\inner{\pi(\lambda)^*U_C^{-1}f, g}
=
e^{-\pi i x^TCx}\,\inner{U_C^{-1}f,\pi(\lambda)g}.
\end{align*}
Taking moduli,
$$
|\inner{f,\pi(S_C\lambda)U_C g}|
=
|\inner{U_C^{-1}f,\pi(\lambda)g}|.
$$
Therefore,
$$
\|C_{U_C g,\Lambda_C} f\|_{\ell^2(\Lambda_C)}^2
=
\sum_{\mu\in\Lambda_C}|\inner{f,\pi(\mu)U_C g}|^2
=
\sum_{\lambda\in\Lambda}|\inner{U_C^{-1}f,\pi(\lambda)g}|^2
=
\|C_{g,\Lambda}(U_C^{-1}f)\|_{\ell^2(\Lambda)}^2.
$$
Taking suprema over $\|f\|_2=1$ and using $\|U_C^{-1}f\|_2=\|f\|_2$, we obtain
$$
\|C_{U_C g,\Lambda_C}\|
=
\|C_{g,\Lambda}\|.
$$
For the synthesis operator, let $c\in\ell^2(\Lambda)$ and define $(c_C)_{\mu}:=c_\lambda$ when $\mu=S_C\lambda$.
Then $c\mapsto c_C$ is an isometric identification $\ell^2(\Lambda)\cong \ell^2(\Lambda_C)$ and
$$
D_{U_C g,\Lambda_C} c_C
=
\sum_{\mu\in\Lambda_C} (c_C)_\mu\,\pi(\mu)U_C g
=
\sum_{\lambda\in\Lambda} c_\lambda\,\pi(S_C\lambda)U_C g.
$$
Using Lemma~\ref{lem:chirp_covariance} again,
$$
\pi(S_C\lambda)U_C g
=
e^{\pi i x^TCx}\,U_C\pi(\lambda)g,
$$
hence
$$
D_{U_C g,\Lambda_C} c_C
=
U_C\left(\sum_{\lambda\in\Lambda} c_\lambda\,e^{\pi i x^TCx}\,\pi(\lambda)g\right).
$$
The unimodular factors $e^{\pi i x^TCx}$ define a diagonal unitary on $\ell^2(\Lambda)$, so
$\|D_{U_C g,\Lambda_C}\|=\|D_{g,\Lambda}\|$ follows by taking operator norms and using unitarity of $U_C$.
Finally, frame bounds are equivalent to bounds for $\|C_{g,\Lambda}f\|^2$ in terms of $\|f\|^2$, so they are preserved.
\end{proof}

\subsubsection{A concrete non-implication: coupling can vary arbitrarily at fixed frame bounds}
The previous proposition shows that, within our chirp-transport class, \emph{frame stability constants do not see the coupling parameter $C$}.
In contrast, the covariance-level defect $\mathsf D$ \emph{does} see $C$ explicitly.
This yields a precise statement: within families with identical frame bounds, one may generate arbitrarily strong
time--frequency coupling.

\begin{proposition}[Arbitrarily large coupling at fixed stability]\label{prop:coupling_vs_bounds}
Let $\Lambda$ be a lattice and let $g,\phi\in C_c^\infty(\R^d)$ with $\|g\|_2=\|\phi\|_2=1$.
Assume $\mathcal G(g,\Lambda)$ is a frame with bounds $A,B$ and assume $\Sigma_{g,\phi}$ exists with $\Sigma_{x\omega}=0$ and $\Sigma_{xx}$ positive definite.
For $t>0$ set $C_t:=tC_0$ for some fixed symmetric matrix $C_0\neq 0$, and define
$$
g_t:=U_{C_t}g,
\qquad
\phi_t:=U_{C_t}\phi,
\qquad
\Lambda_t:=S_{C_t}\Lambda.
$$
Then:
\begin{enumerate}[label=\textup{(\roman*)}]
\item $\mathcal G(g_t,\Lambda_t)$ is a frame with the same bounds $A,B$ for every $t>0$.
\item The coupling defect satisfies $\mathsf D(g_t,\phi_t)=t\,\|\Sigma_{xx}C_0^T\|_{\mathrm{F}}$, hence $\mathsf D(g_t,\phi_t)\to\infty$ as $t\to\infty$.
\end{enumerate}
\end{proposition}

\begin{proof}
Point (i) is exactly Proposition~\ref{prop:analysis_synthesis_conjugacy}.
For (ii), Corollary~\ref{cor:coupling_created} applied to $C=C_t=tC_0$ gives
$$
\mathsf D(g_t,\phi_t)=\|\Sigma_{xx}(tC_0)^T\|_{\mathrm{F}}=t\,\|\Sigma_{xx}C_0^T\|_{\mathrm{F}},
$$
and the factor $\|\Sigma_{xx}C_0^T\|_{\mathrm{F}}$ is strictly positive because $\Sigma_{xx}$ is positive definite and $C_0\neq 0$.
\end{proof}

\begin{remark}[Implication for ``optimal atoms versus optimal frames'']\label{rem:optimal_atoms_vs_frames}
Proposition~\ref{prop:coupling_vs_bounds} shows that, at least within our explicit class,
atom-level coupling/localization descriptors and frame-level stability constants can be \emph{decoupled}:
one may vary coupling arbitrarily without changing $(A,B)$.
This provides a concrete mechanism behind the general question of whether minimizing an atom-level
uncertainty functional necessarily yields the ``best'' frame in global stability terms:
even when the frame bounds are fixed, the geometric structure of phase-space concentration may change substantially.
\end{remark}

\subsection{Robust reconstruction under coefficient noise: deterministic and statistical bounds}\label{subsec:robustness}

\subsubsection{A general noisy-coefficient model}
Let $g,\gamma\in L^2(\R^d)$ and let $\Lambda$ be a lattice.
Assume that $\mathcal G(g,\Lambda)$ is a Bessel system, not necessarily a frame, so that
the analysis operator $C_{g,\Lambda}$ is bounded. Assume also that $\mathcal G(\gamma,\Lambda)$
is a Bessel system, so that the synthesis operator $D_{\gamma,\Lambda}$ is bounded.
We consider the observation model
$$
y = C_{g,\Lambda} f + \eta \in \ell^2(\Lambda),
$$
where $f\in L^2(\R^d)$ is the unknown signal and $\eta\in\ell^2(\Lambda)$ is an additive noise sequence.
The reconstruction with synthesis window $\gamma$ is
$$
\widehat f := D_{\gamma,\Lambda}y
= D_{\gamma,\Lambda}C_{g,\Lambda}f + D_{\gamma,\Lambda}\eta
= \Theta_{g,\gamma,\Lambda}f + D_{\gamma,\Lambda}\eta.
$$

\subsubsection{Deterministic error bounds (exact dual and approximate dual)}
\begin{proposition}[Deterministic robustness bound]\label{prop:deterministic_robustness}
Assume that $\mathcal G(g,\Lambda)$ and $\mathcal G(\gamma,\Lambda)$ are Bessel systems, and let
$B_\gamma$ be a Bessel bound for $\mathcal G(\gamma,\Lambda)$, so that
$$
\|D_{\gamma,\Lambda}\|\le \sqrt{B_\gamma}.
$$
Then, for all $f\in L^2(\R^d)$ and all $\eta\in \ell^2(\Lambda)$,
\begin{equation}\label{eq:deterministic_bound}
\|f-\widehat f\|_2
\le
\|(I-\Theta_{g,\gamma,\Lambda})f\|_2
+
\sqrt{B_\gamma}\,\|\eta\|_{\ell^2}.
\end{equation}
In particular, if the two Bessel systems form an approximately dual pair, namely
$$
\delta:=\|I-\Theta_{g,\gamma,\Lambda}\|<1,
$$
then
\begin{equation}\label{eq:approx_dual_det_bound}
\|f-\widehat f\|_2
\le
\delta\,\|f\|_2
+
\sqrt{B_\gamma}\,\|\eta\|_{\ell^2}.
\end{equation}
If $\gamma$ is an exact dual window of $g$ on $\Lambda$, then
$\Theta_{g,\gamma,\Lambda}=I$ and \eqref{eq:deterministic_bound} reduces to
$$
\|f-\widehat f\|_2
\le
\sqrt{B_\gamma}\,\|\eta\|_{\ell^2}.
$$
If, moreover, $\gamma=\widetilde g=S_{g,\Lambda}^{-1}g$ is the canonical dual window and
$\mathcal G(g,\Lambda)$ has lower frame bound $A_g$, then $B_{\widetilde g}=A_g^{-1}$ and hence
\begin{equation}\label{eq:canonical_dual_noise_bound}
\|f-\widehat f\|_2
\le
\frac{1}{\sqrt{A_g}}\,\|\eta\|_{\ell^2}.
\end{equation}
\end{proposition}
\begin{proof}
From $\widehat f=\Theta f + D_{\gamma,\Lambda}\eta$ we have
$$
f-\widehat f = (I-\Theta_{g,\gamma,\Lambda})f - D_{\gamma,\Lambda}\eta.
$$
Hence, by the triangle inequality and the operator norm bound for $D_{\gamma,\Lambda}$,
$$
\|f-\widehat f\|_2
\le
\|(I-\Theta_{g,\gamma,\Lambda})f\|_2 + \|D_{\gamma,\Lambda}\eta\|_2
\le
\|(I-\Theta_{g,\gamma,\Lambda})f\|_2 + \|D_{\gamma,\Lambda}\|\,\|\eta\|_{\ell^2}
\le
\|(I-\Theta)f\|_2 + \sqrt{B_\gamma}\,\|\eta\|_{\ell^2},
$$
which is \eqref{eq:deterministic_bound}. If $\|I-\Theta\|=\delta$, then
$\|(I-\Theta)f\|_2\le \delta\|f\|_2$, giving \eqref{eq:approx_dual_det_bound}.
If $\gamma$ is an exact dual, then $\Theta=I$ by definition, and the bias term vanishes.
\end{proof}

\subsubsection{Mean-square error under finite noisy measurements}
White noise on an infinite lattice is not an $\ell^2$ random element, so we work with a standard
finite-measurement model.
Let $\Lambda_N\subset\Lambda$ be a finite index set (e.g.\ a box truncation) and consider noisy
measurements
$$
y_\lambda=\inner{f,\pi(\lambda)g}+\eta_\lambda,\qquad \lambda\in\Lambda_N,
$$
where $\{\eta_\lambda\}_{\lambda\in\Lambda_N}$ are complex random variables satisfying
$$
\E[\eta_\lambda]=0,
\qquad
\E[\eta_\lambda\overline{\eta_\mu}]
=
\sigma^2\,\delta_{\lambda\mu},
\qquad \lambda,\mu\in\Lambda_N.
$$

Define the truncated synthesis operator
$$
D_{\gamma,\Lambda_N}c := \sum_{\lambda\in\Lambda_N} c_\lambda\,\pi(\lambda)\gamma,
$$
and the estimator
$$
\widehat f_N:=\sum_{\lambda\in\Lambda_N} y_\lambda\,\pi(\lambda)\gamma
=
\Theta_{g,\gamma,\Lambda_N}f + \sum_{\lambda\in\Lambda_N}\eta_\lambda\,\pi(\lambda)\gamma,
$$
where $\Theta_{g,\gamma,\Lambda_N}:=D_{\gamma,\Lambda_N}C_{g,\Lambda_N}$ is the finite-section mixed operator.

\begin{proposition}[Exact MSE decomposition for uncorrelated coefficient noise]\label{prop:mse_finite}
Assume $\E[\eta_\lambda]=0$ and $\E[\eta_\lambda\overline{\eta_\mu}]=\sigma^2\delta_{\lambda\mu}$ on $\Lambda_N$.
Then
\begin{equation}\label{eq:mse_decomposition}
\E\|f-\widehat f_N\|_2^2
=
\|(I-\Theta_{g,\gamma,\Lambda_N})f\|_2^2
+
\sigma^2\,|\Lambda_N|\,\|\gamma\|_2^2.
\end{equation}
In particular, if $\gamma$ is an exact dual for the finite-section model (i.e.\ $\Theta_{g,\gamma,\Lambda_N}=I$ on the relevant subspace),
then $\E\|f-\widehat f_N\|_2^2=\sigma^2|\Lambda_N|\|\gamma\|_2^2$.
\end{proposition}

{\begin{proof}
Write
$$
f-\widehat f_N
=
(I-\Theta_{g,\gamma,\Lambda_N})f
-
\sum_{\lambda\in\Lambda_N}\eta_\lambda\,\pi(\lambda)\gamma
=:b-n,
$$
where $b$ is deterministic and $n$ is random. Since $\E[\eta_\lambda]=0$ for all
$\lambda$, we have $\E[n]=0$. Hence
$$
\E\|b-n\|_2^2
=
\|b\|_2^2
+
\E\|n\|_2^2
-
2\Rea\,\inner{b,\E[n]}
=
\|b\|_2^2+\E\|n\|_2^2.
$$
It remains to compute $\E\|n\|_2^2$. Using the complex Hilbert-space inner product,
$$
\begin{aligned}
\|n\|_2^2
&=
\left\|
\sum_{\lambda\in\Lambda_N}\eta_\lambda\,\pi(\lambda)\gamma
\right\|_2^2 \\
&=
\sum_{\lambda,\mu\in\Lambda_N}
\eta_\lambda\overline{\eta_\mu}\,
\inner{\pi(\lambda)\gamma,\pi(\mu)\gamma}.
\end{aligned}
$$
Taking expectations and using
$\E[\eta_\lambda\overline{\eta_\mu}]=\sigma^2\delta_{\lambda\mu}$, we obtain
$$
\begin{aligned}
\E\|n\|_2^2
&=
\sum_{\lambda,\mu\in\Lambda_N}
\E[\eta_\lambda\overline{\eta_\mu}]\,
\inner{\pi(\lambda)\gamma,\pi(\mu)\gamma} \\
&=
\sigma^2
\sum_{\lambda\in\Lambda_N}
\inner{\pi(\lambda)\gamma,\pi(\lambda)\gamma} \\
&=
\sigma^2
\sum_{\lambda\in\Lambda_N}
\|\gamma\|_2^2
=
\sigma^2|\Lambda_N|\|\gamma\|_2^2,
\end{aligned}
$$
because each time-frequency shift $\pi(\lambda)$ is unitary. Combining this identity with the
bias term gives \eqref{eq:mse_decomposition}.
\end{proof}}

\subsubsection{Robustness bounds under chirp transport}
{We finally record how the deterministic and statistical robustness bounds behave under chirp
transport. The operator-norm quantities and Bessel bounds are invariant under the unitary
conjugacy, while the finite-section MSE identity is equivariant with respect to the simultaneous
transformation of the signal, windows, and lattice.}

\begin{proposition}[Robustness bounds under chirp transport]\label{prop:robustness_transport}
Let $C=C^T$, $\Lambda_C=S_C\Lambda$, $g_C=U_C g$, $\gamma_C=U_C\gamma$.
\begin{enumerate}[label=\textup{(\roman*)}]
\item (Approximate-dual error) $\|I-\Theta_{g_C,\gamma_C,\Lambda_C}\|=\|I-\Theta_{g,\gamma,\Lambda}\|$.
\item (Deterministic bound) The Bessel bound of $\gamma$ equals the Bessel bound of $\gamma_C$, hence
the constant $\sqrt{B_\gamma}$ in \eqref{eq:deterministic_bound} is unchanged.
\item {(Finite-section MSE equivariance) Let $\Lambda_N\subset\Lambda$ be finite and set
$\Lambda_{C,N}:=S_C\Lambda_N$. For $f\in L^2(\R^d)$, define
$$
\operatorname{MSE}_{g,\gamma,\Lambda_N}(f)
:=
\|(I-\Theta_{g,\gamma,\Lambda_N})f\|_2^2
+
\sigma^2|\Lambda_N|\|\gamma\|_2^2 .
$$
Then
$$
\operatorname{MSE}_{g_C,\gamma_C,\Lambda_{C,N}}(U_Cf)
=
\operatorname{MSE}_{g,\gamma,\Lambda_N}(f).
$$
Equivalently, for an arbitrary $h\in L^2(\R^d)$,
$$
\operatorname{MSE}_{g_C,\gamma_C,\Lambda_{C,N}}(h)
=
\operatorname{MSE}_{g,\gamma,\Lambda_N}(U_C^{-1}h).
$$}
\end{enumerate}
\end{proposition}

\begin{proof}
(i) is exactly Corollary~\ref{cor:approx_dual_transport} proved earlier.

(ii) By Proposition~\ref{prop:analysis_synthesis_conjugacy}, $\|D_{\gamma_C,\Lambda_C}\|=\|D_{\gamma,\Lambda}\|$,
hence the optimal Bessel bound is preserved, so $\sqrt{B_{\gamma_C}}=\sqrt{B_\gamma}$.

(iii) {For the finite section, note first that
$$
\|\gamma_C\|_2=\|\gamma\|_2
$$
because $U_C$ is unitary, and
$$
|\Lambda_{C,N}|=|\Lambda_N|
$$
because $S_C$ is a bijection between the finite sets. Hence the noise contribution
$$
\sigma^2|\Lambda_{C,N}|\|\gamma_C\|_2^2
$$
equals
$$
\sigma^2|\Lambda_N|\|\gamma\|_2^2.
$$
Moreover, by Proposition~\ref{prop:mixed_conjugacy}, applied to the finite set $\Lambda_N$,
$$
\Theta_{g_C,\gamma_C,\Lambda_{C,N}}
=
U_C\,\Theta_{g,\gamma,\Lambda_N}\,U_C^{-1}.
$$
Therefore, for every $f\in L^2(\R^d)$,
$$
\begin{aligned}
\|(I-\Theta_{g_C,\gamma_C,\Lambda_{C,N}})U_Cf\|_2
&=
\|U_C(I-\Theta_{g,\gamma,\Lambda_N})f\|_2 \\
&=
\|(I-\Theta_{g,\gamma,\Lambda_N})f\|_2.
\end{aligned}
$$
Combining equality of the bias terms for the corresponding signals with equality of the noise
terms yields
$$
\operatorname{MSE}_{g_C,\gamma_C,\Lambda_{C,N}}(U_Cf)
=
\operatorname{MSE}_{g,\gamma,\Lambda_N}(f).
$$
The equivalent identity for a general $h$ follows by taking $h=U_Cf$, i.e. $f=U_C^{-1}h$.}
\end{proof}
{The finite-section MSE identity above is an equivariance statement. It compares the original
system acting on $f$ with the chirp-transported system acting on the corresponding signal
$U_Cf$. It does not assert, in general, that
$$
\operatorname{MSE}_{g_C,\gamma_C,\Lambda_{C,N}}(f)
=
\operatorname{MSE}_{g,\gamma,\Lambda_N}(f)
$$
for the same unchanged signal $f$. Indeed,
$$
\|(I-\Theta_{g_C,\gamma_C,\Lambda_{C,N}})f\|_2
=
\|(I-\Theta_{g,\gamma,\Lambda_N})U_C^{-1}f\|_2,
$$
which need not equal
$$
\|(I-\Theta_{g,\gamma,\Lambda_N})f\|_2.
$$
}
\begin{remark}[Summary]
Within the explicit chirp-transport class of this paper, non-separability can be quantified by a covariance-level defect
$\mathsf D$ that grows linearly with the coupling parameter $C$, while frame bounds and robustness constants remain invariant.
This isolates a precise sense in which non-separability changes the geometry of time--frequency concentration without changing
the global stability/robustness constants inherited from the separable model.
\end{remark}

\section{Concluding remarks}

{In this paper we proposed a constructive route to chirped Gabor windows on $\R^d$ based on three ingredients:
separable exact dual pairs, chirp covariance, and transport of duality. The main outcome is an explicit class
of higher-dimensional chirped tensor-product windows on lattices of the form
$$
\Lambda_{A,B,C}
=
\begin{pmatrix}
A&0\\
CA&B
\end{pmatrix}\Z^{2d},
\qquad C=C^T,
$$
for which exact dual frames, compact support, smoothness, and approximation of the canonical dual are inherited
from the separable setting. Under the off-diagonal condition on $C$ stated in
Proposition~\ref{prop:nonseparable}, these windows are genuinely non-separable.}

The scope of the paper is intentionally constructive. In this sense, the present work is closer in spirit to the explicit-window philosophy of \citep{DGM86} and to the geometric constructive program of \citep{PRW12} than to a full classification theory. At the same time, the control of exact and approximate duals connects naturally with the lines developed in \citep{Sto22,CJKK18}. In relation to earlier non-separable-lattice schemes, especially the shear-based formulations
of \citet{VanLeestBastiaans2000} and the discrete implementations in
\citep{VanLeest1999,VanLeestBastiaans2004}, the contribution here is therefore not the
use of shear geometry per se. Rather, it is the explicit Hilbert-space transport of compactly
supported tensor-product dual pairs through a chirp, together with the preservation of
exact and approximate duality data and the off-diagonal criterion guaranteeing genuine
functional non-separability of the resulting windows.

Several directions remain open.

First, it would be desirable to replace the purely chirp-type coupling by more general metaplectic transformations while still retaining a usable control on support, smoothness, or anisotropic decay. This would enlarge the lattice classes accessible by explicit constructions and would move the present approach closer to the geometric flexibility already visible in \citep{PRW12}.

Second, one would like to understand to what extent the present transport principle can be combined with higher-dimensional geometric constructions of smooth compactly supported windows in order to obtain dual pairs beyond the current lower block-triangular setting. It would also be interesting to investigate whether corresponding dual systems enjoy stability properties in modulation or amalgam spaces, in the direction suggested by \citep{BCKOR13}.

Third, the constructive program developed here should eventually be compared with the genuinely different Gaussian program. For Gaussian windows, already in dimension $d\geq 2$ the frame property is tied to sampling in Bargmann--Fock spaces of several complex variables and is no longer governed by density alone \citep{Gro11}. A natural long-term objective is therefore to investigate whether some part of the present dual-transport philosophy can be merged with the multivariate Gaussian/sampling viewpoint. From a complementary perspective, one may also ask whether lattice-shape optimization phenomena, such as those studied for separable Gaussian lattices in \citep{FS17}, admit constructive counterparts in the present chirp-sheared setting.

\section*{Acknowledgment}
A. Mazzoccoli and P. Vellucci are members of the INdAM Research group GNCS. We are sincerely grateful to Prof. Laura De Carli for her valuable support and insightful guidance.

\bibliographystyle{elsarticle-harv}
\bibliography{example}

\end{document}